\documentclass[a4paper,11pt]{article}

\usepackage{mathpazo}

\usepackage[active]{srcltx}
\usepackage[utf8]{inputenc}
\usepackage[english]{babel}
\usepackage{graphicx}
\usepackage[T1]{fontenc}
\usepackage{amsmath}
\usepackage{amssymb} 
\usepackage{amsthm}
\usepackage{mathabx}
\usepackage{bm} 
\usepackage{eqnarray}
\usepackage{tikz}
\usepackage[numbers]{natbib}
\usepackage{array}
\usepackage{esint}
\usepackage{enumitem}
\usepackage{makecell}
\usepackage[left=2.05cm,right=2.05cm,top=2.95cm,bottom=2.8cm]{geometry}
\usepackage{algorithm}
\usepackage{algpseudocode}
\usepackage{comment}
\usepackage{units}
\usepackage{nccmath}
\usepackage[tableposition=top]{caption}

\usepackage[pdfborder={0 0 0}]{hyperref}

\numberwithin{equation}{section} 

\newcommand{\R}{\mathbb{R}}

\newcommand{\U}{\mathcal{U}} 

\newcommand{\E}{\mathbb{E}}

\newcommand{\1}{\mathbb{1}} 
\newcommand{\sX}{\mathcal{X}}
\newcommand{\sP}{\mathcal{P}}

\def\sO {\mathcal{O}}

\newcommand{\sH}{\mathcal{H}}

\renewcommand{\d}{\mathrm{d}} 

\newcommand{\Id}{\mathrm{Id}}
\newcommand{\KL}{\mathrm{KL}}

\newcommand{\gap}{\mathrm{Gap}}

\newtheorem{theorem}{Theorem}[section]
\newtheorem{prop}[theorem]{Proposition}
\newtheorem{corollary}[theorem]{Corollary}
\newtheorem{lemma}[theorem]{Lemma}
\newtheorem{definition}[theorem]{Definition}

\newtheorem{asmp}{Assumption}
\newtheorem*{asmp*}{Assumption}

\theoremstyle{remark}
\newtheorem{remark}[theorem]{Remark}

\begin{document}

\title{Entropy contraction of the Gibbs sampler under log-concavity}
\author{Filippo Ascolani\thanks{Duke University, Department of Statistical Science, Durham, NC, United States (filippo.ascolani@duke.edu)}, \, Hugo Lavenant\thanks{Bocconi University, Department of Decision Sciences and BIDSA, Milan, Italy (hugo.lavenant@unibocconi.it)} \, and Giacomo Zanella\thanks{Bocconi University, Department of Decision Sciences and BIDSA, Milan, Italy (giacomo.zanella@unibocconi.it)\\HL acknowledges the support of the MUR-Prin 2022-202244A7YL funded by the European Union - Next Generation EU. GZ acknowledges support from the European Research Council (ERC), through StG ``PrSc-HDBayLe''
grant ID 101076564.}}
\date{\today}

\maketitle

\begin{abstract}
The Gibbs sampler (a.k.a. Glauber dynamics and heat-bath algorithm) is a popular Markov Chain Monte Carlo algorithm which iteratively samples from the 
conditional distributions of a probability measure $\pi$ of interest. 
Under the assumption that $\pi$ is strongly log-concave, we show that the random scan Gibbs sampler contracts in relative entropy and provide a sharp characterization of the associated contraction rate. 
Assuming that evaluating conditionals is cheap compared to evaluating the joint density, our results imply that the number of full evaluations of $\pi$ needed for the Gibbs sampler to mix grows linearly with the condition number and is independent of the dimension.
If $\pi$ is non-strongly log-concave, the convergence rate in entropy degrades from exponential to polynomial. Our techniques are versatile and extend to Metropolis-within-Gibbs schemes and the Hit-and-Run algorithm. 
A comparison with gradient-based schemes and the connection with the optimization literature are also discussed.
\end{abstract}

\section{Introduction}\label{sec:intro}

\subsection*{MCMC and Gibbs Sampling}

Sampling from a target probability distribution $\pi$ is a key task arising in many scientific areas \citep{liu2001monte}, which becomes particularly challenging when $\pi$ is high-dimensional and not analytically tractable.
This task is often accomplished with Markov Chain Monte Carlo (MCMC) algorithms \citep{B11}, which simulate a Markov chain $(X^{(n)})_{n=0,1,2,\dots}$ whose distribution converges to $\pi$ as $n \to \infty$. 
When studying MCMC schemes, a fundamental step is to determine how quickly this convergence occurs, i.e. how long the algorithm should be run for.

In this paper we focus on the Gibbs sampler (GS), which is one of the most canonical and popular examples of MCMC algorithms \citep{geman1984stochastic,C92}. It is a coordinate-wise method, which works in the setting where the state space $\sX$ of $\pi$ decomposes as a product space $\sX = \sX_1 \times \sX_2 \times \ldots \times \sX_M$. For an index $m$, we write $\pi(\cdot | X_{-m}) \in \sP(\sX_m)$ for the ``full conditional'', that is the conditional distribution of $X_m$ given $X_{-m} = (X_1, \ldots, X_{m-1}, X_{m+1}, \ldots, X_M)$ under $X\sim \pi$. 
We focus on the \emph{random scan} version of GS, which works by selecting a coordinate at random for every iteration and updating it from the corresponding full conditional. This is described in Algorithm \ref{alg:rs-gs}.  
The resulting Markov chain is $\pi$-invariant and the law of $X^{(n)}$ converges to $\pi$ as $n\to\infty$ under mild regularity assumptions, see e.g.\  \citep{roberts1994simple} for weak conditions in the case $\sX=\R^d$.
Crucially, GS only requires the user to simulate from $\pi(\cdot|X_{-m})$, which is often a much easier task than simulating directly from $\pi$, either because $\sX_m$ is low-dimensional (e.g.\ $\sX_m=\R$ while $\sX=\R^d$ with $d\gg 1$) or because $\pi(\cdot|X_{-m})$ belongs to some known class of analytically tractable distribution; see Section \ref{sec:implementation} for more details.

\begin{algorithm}[htbp]
\begin{algorithmic}
\State Initialize $X^{(0)}=(X^{(0)}_1,\dots,X^{(0)}_M)\in \sX$
\For{$n=0,1,2,\dots,$}
    \State Sample $m_n \in \{1,\dots,M\}$ uniformly at random
    \State Sample $X^{(n+1)}_{m_n} \in \sX_{m_n}$ from the distribution $\pi(\cdot | X^{(n)}_{-m_n})$
    \State Set $X^{(n+1)}_i= X^{(n)}_i$ for $i\neq m_n$
   \EndFor
\end{algorithmic}
\caption{(Random Scan Gibbs Sampler)
\label{alg:rs-gs}}
\end{algorithm}

\subsection*{Log-concave sampling}

Convergence of MCMC methods is often studied under log-concavity assumptions, meaning that $\sX=\R^d$ and $\pi(\d x) = \exp(-U(x)) \d x$ with $U$ (often called potential of $\pi$) being a convex function; see Section~\ref{section:convexity} for more detailed assumptions on $U$. 
In the last decade or so, the log-concave sampling literature has witnessed remarkable advancements in the quantitative convergence analysis of MCMC algorithms, especially gradient-based ones such as Langevin \cite{D17, DM17,eberle2019couplings,D19,wu2022minimax} or Hamiltonian \cite{bou2020coupling,chen2020fast} Monte Carlo. See  e.g.\ \citep{chewi2023log} for a monograph describing progress on the topic. 
On the contrary, despite its popularity and widespread use, very few results that quantify the convergence speed of GS under log-concavity are available, see Section \ref{sec:lit_review} for a review. 
The present work intends to fill this substantial gap, providing explicit and sharp results for GS targeting log-concave distributions.   
Our main result is summarized in the next paragraph, and its computational implications are briefly discussed afterwards.

\subsection*{Main result}
In the above notation, take $\sX_m = \R^{d_m}$ for $m=1, \ldots, M$, so that $\sX = \R^d$ with $d=d_1 + \ldots + d_M $. 
In particular, $M=d$ if all blocks are of dimension one, i.e.\ $d_m=1$ for all $m$.
Let $P^\mathrm{GS}$ be the transition kernel associated to Algorithm \ref{alg:rs-gs} with $\pi = \exp(-U)$, and assume $U$ is $\lambda$-convex and $L$-smooth (see Sections \ref{subsec:GS_Pm} and \ref{section:convexity} for formal definitions). 
Then our key bound reads
\begin{equation}
\label{eq:main_intro}
\KL(\mu P^\mathrm{GS} | \pi) \leq \left( 1 - \frac{1}{\kappa M} \right) \KL(\mu | \pi),
\end{equation}
where $\kappa = L/\lambda$ is the condition number of $U$, $\KL$ stands for the Kullback-Leibler divergence or relative entropy, and $\mu$ is an arbitrary distribution on $\R^d$ (with $\KL(\mu | \pi)<\infty$, otherwise the bound is trivial). 
Actually, our main result (Theorem~\ref{theo:contraction_KL_main}) is finer, since it allows for separable non-smooth parts in the potential and it depends on a ``coordinate-wise'' condition number $\kappa^*$, which leverages the block structure of the problem.
Explicit computations on Gaussian distributions~\cite{amit1996convergence} show that the resulting bound is tight up to small constant factors, see Section~\ref{subsec:tightness}.  

Denoting by $\mu^{(n)}$ the law of the Markov chain after $n$ steps, the bound in~\eqref{eq:main_intro} directly implies a geometric decay of $\KL(\mu^{(n)} | \pi)$ towards zero, and in particular it implies that the mixing time of GS, that is, the number of iterations needed to reach an error (in $\mathrm{TV}$ or $\KL$) from $\pi$ of $\epsilon$, is of order
$\sO(M\kappa\log(1/\epsilon))$, 
see Corollary \ref{crl:GS_mix_time} for more details.  
It also implies a bound on the spectral gap of $P^\mathrm{GS}$ (Corollary \ref{crl:variance_inequality}), and it admits some variational representations in terms of marginal and conditional entropies and variances that can be of independent interest; see \eqref{eq:functional_main} and \eqref{eq:variance_inequality}.

\begin{remark}
A popular variant of Algorithm~\ref{alg:rs-gs} is the deterministic scan version, where $X_m$ is updated cyclically for $m=1,\ldots,M$ at each iteration. 
In this work we focus on the random scan case since it enjoys better theoretical properties.  
A parallel can be drawn with the literature on convex optimization, where it is well-known that random scan coordinate descent algorithms have better worst-case complexity than deterministic scan ones \cite{W15}. 
We expect similarly that deterministic scan GS behaves worse than random scan GS in terms of worst-case performance for log-concave targets, even if we are not aware of rigorous results in that direction.
\end{remark}

\subsection*{Computational implications and comparison with gradient-based schemes}

In many relevant settings, the computational cost of one iteration of Algorithm \ref{alg:rs-gs} (which only performs one update from a single ``full conditional'' distribution) is $M$ times smaller than the cost of evaluating the joint log-density $U$ (or its gradient) once; see Section~\ref{sec:comparison_gradient} for details. 
In those cases, our result implies that GS produces an $\epsilon$-accurate sample with a total computational cost comparable to 
\[
\sO(\kappa\log(1/\epsilon))
\]
target evaluations. 
This quantity scales linearly with $\kappa$, and, remarkably, \emph{does not depend on the dimension $d$}. 
By contrast, in the log-concave setting, gradient-based methods such as Langevin or Hamiltonian Monte Carlo 
require a number of target or gradient evaluations for each $\epsilon$-accurate sample that grows with both $\kappa$ and $d$. 
Intuitively, the dimension-free performances of GS as opposed to gradient-based MCMC schemes are related to the fact that the former does not incur any discretization error, while the latter do (due e.g.\ to Euler-Maruyama or other type of approximations).
See Sections \ref{sec:comparison_gradient} and \ref{sec:comparison_optimization} for more details.
Perhaps counter-intuitively, the same dimension-free behavior holds true also for coordinate-wise MCMC methods that employ a Metropolis-Hastings correction inside each full conditional update, such as Metropolis-within-Gibbs schemes: see Section~\ref{sec:MwG} and in particular Theorem \ref{thm:gap_MwG_RWM}.

\subsection*{Extension to Hit-and-Run and variants}

Our main result also allows to analyze other MCMC algorithms which rely on conditional updates (or equivalently on ``random projections"). One notable example is the Hit-and-Run sampler \citep{vempala2005geometric}, described in Algorithm \ref{alg:h_and_run} and whose transition kernel we denote as $P^\mathrm{HR}$.
In words, $P^\mathrm{HR}$ picks uniformly at random a line passing through the current point $x$ and samples the next point from $\pi$ restricted to that line. 

\begin{algorithm}[htbp]
\begin{algorithmic}
\State Initialize $X^{(0)}\in \R^d$
\For{$n=0,1,2,\dots,$}
    \State Sample a direction $v$ uniformly at random from $\mathbb{S}^{d-1}:=\{z\in\R^d\,:\,\|z\|=1\}\subseteq\R^d$
    \State Sample $s\in \R$ with density $\pi_v(s)$ proportional to $\pi\left(X^{(n)}+sv\right)$
    \State Set $X^{(n+1)} = X^{(n)}+sv$
   \EndFor
\end{algorithmic}
\caption{(Hit-and-Run Sampler)
\label{alg:h_and_run}}
\end{algorithm}

With $\pi = \exp(-U)$ and $U$ being $\lambda$-convex and $L$-smooth, we prove the following entropy contraction result (see Theorem~\ref{theo:HR_ell})
\[
\KL(\mu P^{\mathrm{HR}} | \pi) \leq \left( 1 - \frac{1}{\kappa d} \right) \KL(\mu|\pi)\,,
\]
with $\kappa=L/\lambda$ as above. 
The result is a consequence of \eqref{eq:main_intro} after noting that $P^\mathrm{HR}$ can be written as a mixture of GS kernels over uniformly rotated coordinate axes.
We discuss more in details Hit-and-Run and other more general sampling schemes based on random projections in Section \ref{sec:extension_HR}.

\subsection*{Proof technique}
The core of the technical work is to prove~\eqref{eq:main_intro}. The strategy drastically differs from the analysis of gradient-based MCMC methods: while for the latter a common approach interprets them as discretizations of continuous diffusion processes, the Gibbs sampler intrinsically lives in discrete time, with potentially large jumps between $X^{(n)}$ and $X^{(n+1)}$ (i.e.\ the associated Markov operator is intrinsically ``non-local''). Our approach is based on the variational characterization of the Gibbs kernel via $\KL$ (see Lemma~\ref{lm:variational_characterization_Pm}): leveraging it, we use well-crafted transport maps, inspired by the literature on coordinate-wise optimization, to provide lower bounds on the decay of $\KL$ when applying the GS kernel.  
A key step in our proof is to rely on triangular transport maps (specifically the Knothe-Rosenblatt map), which allow for a neat decomposition of the entropy that relates maps jointly moving all coordinates with maps moving one coordinate at a time (see Lemma \ref{lm:entropy_triangular_eq}).
Contrarily to many works on gradient-based sampling, our proof technique makes no direct use of optimal transport maps or Wasserstein geometry (see Remark~\ref{rmk:why_not_OT} for related discussion). 

\subsection*{Structure of the paper}
After introducing some notations and the required assumptions on the potential $U$ (Section \ref{sec:notation_background}), we state our main result (Section~\ref{sec:main}). Its proof is postponed until Section~\ref{sec:proof_main}. Before that, we review the implications of entropy contraction in terms of mixing time, spectral gaps and Brascamp-Lieb inequalities (Sections \ref{sec:mix_times} and \ref{sec:spectral_gap}), the tightness of our bound (Section~\ref{subsec:tightness}) and discuss the computational implications of our main result 
(Section~\ref{section:computational_cost}). Section \ref{sec:non_strongly_convex} analyzes the non-strongly convex case, when $\lambda = 0$.
Finally, we move to extensions of our results to Metropolis-within-Gibbs schemes (Section~\ref{sec:MwG}) as well as Hit-and-Run and other variants (Section~\ref{sec:extension_HR}).

\subsection{Related works}\label{sec:lit_review}

Due to its popularity across different scientific communities, the GS (also known as Glauber dynamics, heat-bath and coordinate hit-and-run algorithm) has received a lot of attention over the years. 
Here we briefly discuss some of the main lines of work on the topic, even if an exhaustive review is well beyond the scope of the section.

Explicit results are available for specific and analytically tractable families of distributions, such as Gaussians \cite{amit1996convergence, roberts1997updating} or other fully conjugate exponential families \cite{D08}.
In those cases, a full diagonalization and spectral analysis is possible using appropriate multivariate orthogonal polynomials, see e.g.\ \cite{khare2009rates}. We also mention the recent work \cite{caputo2024entropy}, which derives explicit tight bounds for the entropy contraction coefficient for a class of Gaussian measures.
Other works have focused on specific problems arising in high-dimensional Bayesian statistics, e.g. hierarchical \citep{Q19, Y17,ascolani2024dimension}  or regression \citep{Q22} models, possibly combining the analysis with random data-generating assumptions and statistical asymptotics.

In general settings, the convergence speed of GS has been related to the so-called Dobrushin interdependence coefficients \citep{dobrushin1970prescribing}, see e.g.\ \cite{dyer2008dobrushin,wang2014convergence}. This approach builds upon a long line of work dealing with transport and log-Sobolev inequalities for Gibbs measures, see e.g.\ \cite{zegarlinski1992dobrushin} as well as Section 11 of the review \cite{gozlan2010transport} and the references therein. 
More recently, spectral independence \cite{anari2021spectral} and stochastic localization \citep{chen2025localization} techniques were successfully applied to bound GS mixing times for various discrete models, such as spin systems, see e.g.\ the review in \cite{caputo2023lecture}.
However, we are not aware of successful application of such techniques to study GS targeting general log-concave distributions (see Remark \ref{rmk:prox_sampl} for some exceptions).

Explicit results on coordinate-wise MCMC methods for log-concave distributions include \cite{tong2020mala}, who studies Metropolis-adjusted coordinate-wise MCMC methods for sparse dependence structures under strong conditions on the target potential (e.g.\ bounded derivatives)
and \cite{ding2021random,ding2021underdamped}, who study coordinate-wise (unadjusted) Langevin Monte Carlo schemes. The latter works provide explicit bounds on mixing times in Wasserstein distance under log-concavity, which however have significantly worse dependence on both $d$, $\kappa$ and $\epsilon$ relative to \eqref{eq:main_intro}, arguably due to the presence of Euler-Maruyama discretization error, and lead to a computational complexity roughly similar to standard Langevin Monte Carlo, see e.g.\ \cite{ding2021langevin} for more discussion.

GS is also widely used to sample from uniform distributions on convex bodies, and explicit bounds on their mixing times have been recently established in that contexts \cite{laddha2023convergence} (which scale as polynomial of higher degrees with respect to dimensions and are usually based on conductance bounds).
However, this context is quite different from ours, and it is unclear how much the associated results and techniques are related to the ones employed herein.

Finally, beyond sampling, recent works \citep{arnese2024convergence, LZ2024CAVI} analyze coordinate ascent variational inference (CAVI) algorithms, which compute a product-form approximation of the distribution of interest, under log-concavity assumptions. 
Some of the techniques and computations of the present work are inspired by~\cite{LZ2024CAVI} from the last two authors.
Relative to the CAVI case, however, the GS one is significantly more complex due to the fact that $\mu^{(n)}$ does not factorize across coordinates, i.e.\ the space $\sP(\sX)$ where GS operates does not have a product structure, contrary to the space $\sP(\sX_1)\times\dots\times\sP(\sX_M)$ where CAVI operates. 
This implies that, for example, GS cannot be directly interpreted as a coordinate descent algorithm in the space of measures. 

\section{Notation and assumptions}\label{sec:notation_background}
Throughout the paper we take $\sX_m = \R^{d_m}$ for $m=1, \ldots, M$ so that 
$\sX=\R^d$ with $d = d_1 + \ldots + d_M$. For a point $x=(x_1, \ldots, x_M)\in \R^d$, we write $x_{-m} = (x_1, \ldots, x_{m-1}, x_{m+1}, \ldots, x_M)$, which is an element of $\sX_{-m} = \bigtimes_{i \neq m} \sX_i$. For $x,y\in\R^d$, we also write 
\begin{equation*}
(y_m, x_{-m}) = (x_1, \ldots, x_{m-1}, y_m, x_{m+1}, \ldots, x_M)
\end{equation*}
for the vector $x$ whose $m$-th component has been replaced by the $m$-th component of $y$. 

We denote the Borel $\sigma$-algebra over $\R^d$ as $\mathcal{B}(\R^d)$, the associated space of probability distributions as $\sP(\R^d)$, and the space of square integrable functions with respect to $\pi$ as $L^2(\pi)$. 
For $\mu\in\sP(\R^d)$ and $X\sim \mu$, we write $\mu_{-m}\in\sP(\sX_{-m})$ for the marginal distribution of $X_{-m}$ and $\mu(\cdot | x_{-m})\in\sP(\sX_m)$ for the conditional distribution of $X_m$ given $X_{-m}=x_{-m}$. 
We denote the push-forward of a measure $\mu$ by a map $\phi$ as $\phi_\# \mu$, it satisfies $\phi_\# \mu(A) = \mu\left(\phi^{-1}(A)\right)$ for any $A \in \mathcal{B}(\R^d)$.
When a measure has a density with respect to the Lebesgue measure $\d x$, we use the same letter to denote the measure and its density.
Finally, when $P : \R^d\times \mathcal{B}(\R^d) \to [0,1]$ is a Markov kernel, $\mu\in\sP(\R^d)$ and $f:\R^d\mapsto \R$, we define $\mu P\in\sP(\R^d)$ as $\mu P (A) = \int P(x,A) \, \mu(\d x)$ and $Pf:\R^d\mapsto \R$ as $Pf(x) = \int f(y) \, P(x,\d y)$.

Recall that the Kullback-Leibler (KL) divergence between $\mu, \nu \in \sP(\R^d)$ is defined as 
\begin{equation*}
\KL(\mu | \nu) = \begin{cases}
\displaystyle{\int_{\R^d}\log \left(\frac{\d \mu}{\d \nu}(x) \right) \mu(\d x) } & \text{if } \mu \ll \nu, \\
+ \infty & \text{otherwise.}
\end{cases}
\end{equation*}
The functional $\KL$ is non-negative, jointly convex and vanishes only when $\mu = \nu$~\cite[Section 9.4]{AGS}. 
The additivity properties of the logarithm give the following classical decomposition~\cite[Appendix A]{leonard2013}:
for any $m =1,\ldots, M$,  
\begin{equation}\label{eq:KL_chain_rule}
\KL(\mu | \nu) = \KL(\mu_{-m} | \nu_{-m}) + \E_{X_{-m} \sim \mu_{-m}} \left( \KL(\mu(\cdot | X_{-m} ) | \nu(\cdot | X_{-m} )) \right).
\end{equation} 

\subsection{The Gibbs sampler Markov kernel}
\label{subsec:GS_Pm}

For $m = 1, \ldots, M$, we write $P_m : \R^d \times \mathcal{B}(\R^d) \to [0,1]$ for the Markov kernel defined as 
\begin{equation*}
P_m(x,A) = \int_{\sX_m} \1_{A}(y_m, x_{-m}) \, \pi(\d y_m | x_{-m} ),
\end{equation*}
with $\1_A(\cdot)$ denoting the indicator function of the set $A$. Intuitively, $P_m(x,\cdot)$ keeps the component of $x$ on $\sX_{-m}$ unchanged while the component on $\sX_m$ is drawn according to the full conditional $\pi(\cdot | x_{-m} )$. This is summarized in the next well-known lemma whose proof we omit.   

\begin{lemma}
\label{lm:interpretation_Pm}
For any $m=1,\ldots, M$ the following holds.
\begin{enumerate}
\item[(a)] Let $\mu \in \sP(\R^d)$ and $\nu = \mu P_m$. Then $\mu_{-m} = \nu_{-m}$ and $\nu( \cdot | x_{-m}) = \pi( \cdot | x_{-m})$ for $\mu_{-m}$-almost every $x_{-m} \in \sX_{-m}$.
\item[(b)] If $f \in L^2(\pi)$ and $X \sim \pi$ then $P_m f(x) = \E(f(X) | X_{-m}=x_{-m})$.
\end{enumerate}
\end{lemma}

 The next lemma, which is an immediate consequence of~\eqref{eq:KL_chain_rule}, the non-negativity of $\KL$ and part (a) of Lemma \ref{lm:interpretation_Pm}, provides a variational characterization of $P_m$ in terms of the relative entropy.

\begin{lemma}[Variational characterization]
\label{lm:variational_characterization_Pm}
Let $\mu \in \sP(\R^d)$ and $\KL(\mu|\pi) < + \infty$. Then for any $\nu \in \sP(\R^d)$ with $\nu_{-m} = \mu_{-m}$ we have
\begin{equation*}
\KL(\mu P_m | \pi)=\KL(\mu_{-m} | \pi_{-m})  \leq \KL(\nu | \pi).
\end{equation*}
\end{lemma}

As described in Algorithm~\ref{alg:rs-gs}, random scan GS corresponds to choosing the index $m$ at random at every iteration. At the level of the Markov kernel, it translates into a mixture of the $P_m$'s. 
The GS Markov kernel is thus defined as
\begin{equation}\label{eq:def_GS}
P^\mathrm{GS} = \frac{1}{M} \sum_{m=1}^M P_m. 
\end{equation}

\begin{remark}[Invariance of GS under coordinate-wise transformations]
\label{rmk:idp_parametrization}
Let $\phi :\R^d\to \R^d$ be an invertible transformation which preserves the block structure, that is, $\phi_m(x)$ depends only on $x_m$ for $m=1, \ldots, M$. In other words, $\phi(x) = (\phi_1(x_1), \ldots, \phi_M(x_M))$. Let $\hat{P}^\mathrm{GS}$ be the GS kernel with target distribution $\hat{\pi} = \phi_\# \pi$. Then, it is not difficult to check that for any $\mu \in \sP(\R^d)$ we have $\phi_\# \left( \mu P^\mathrm{GS} \right) = (\phi_\# \mu) \hat{P}^\mathrm{GS}$. Writing $\hat{\mu} = \phi_\# \mu$, as the $\KL$ is invariant with respect to invertible transformations, we deduce
\begin{equation*}
\KL(\mu | \pi) = \KL(\hat{\mu} | \hat{\pi}), \qquad \KL(\mu P^\mathrm{GS} | \pi) = \KL( \hat{\mu} \hat{P}^\mathrm{GS} | \hat{\pi} ). 
\end{equation*}
Thus, we can apply any invertible change of variables $\phi :\R^d\to \R^d$, as long as it preserves the block structure, and the convergence speed of GS in $\KL$ is the same for $\pi$ and $\hat{\pi}$. 
\end{remark}

\subsection{Assumptions on the potential}\label{section:convexity}
Our main working assumption is that $\pi$ is a log-concave distribution. We call $U : \R^d \to \R$ the negative log density of $\pi$, that is, $\pi(\d x) = \exp(-U(x)) \d x$, and we assume that $U$ is convex. To get quantitative rates, we need quantitative assumptions related to the convexity of the potential. We start by recalling the classical notion of smoothness and strong convexity widely used in convex optimization~\cite[Chapter 2]{nesterov2018lectures} and log-concave sampling \cite{chewi2023log}.

\begin{asmp}
\label{asmp:convex_smooth}
The measure $\pi \in \sP(\R^d)$ has density $\exp(-U)$ and $U : \R^d \to \R$ is
\begin{itemize}
\item[(i)] $L$-smooth, in the sense that $\nabla U : \R^d \to \R^d$ is $L$-Lipschitz;
\item[(ii)] $\lambda$-convex, in the sense that $x \mapsto U(x) - \frac{\lambda}{2} \|x \|^2$ is a convex function. 
\end{itemize}
We assume $\lambda>0$ and write $\kappa = L/\lambda$ for the condition number of $U$.
\end{asmp}

The condition number $\kappa$ is independent from the choice of the orthogonal basis of $\R^d$ and larger $\kappa$ indicates that $\pi$ is more anisotropic and harder to sample from. However, GS dynamics are heavily dependent on the coordinate system and, as highlighted in Remark~\ref{rmk:idp_parametrization}, vastly different scales do not necessarily slow down convergence of GS as long as they are located in the same block. Following the literature on coordinate-wise optimization \cite{nesterov2012efficiency,RT14}, we will use a refined condition number, which depends on the block structure of the problem. 
Also, we will relax Assumption \ref{asmp:convex_smooth} by allowing for separable non-smooth terms in the potential $U$.
Below we write $\nabla_m U : \R^d \to \R^{d_m}$ for the gradient of $U$ with respect to the $m$-th coordinate $x_m\in\sX_m$.   

\begin{asmp}
\label{asmp:convex_smooth_block}
The measure $\pi \in \sP(\R^d)$ has density $\exp(-U)$ and $U : \R^d \to \R$ decomposes as 
\begin{equation*}
U(x) = U_0(x) + \sum_{m=1}^M U_m(x_m)\,,
\end{equation*}
with $U_m : \sX_m \to \R$ being a convex function, for each $m=1, \ldots, M$. The function $U_0$ is of class $C^1$ and:
\begin{itemize}
\item[(i)] 
for every $m =1, \ldots, M$ and $\bar{x}_{-m} \in \sX_{-m}$, the function $x_m \mapsto U_0(x_m, \bar{x}_{-m})$ is $L_m$-smooth, in the sense that $x_m \mapsto \nabla_m U_0(x_m,\bar{x}_{-m})$ is $L_m$-Lipschitz;
\item[(ii)] the function $U_0$ is $\lambda^*$-convex for the distance $\| \cdot \|_L$ defined as $\| x \|_L^2 = \sum_m L_m \| x_m \|^2$, in the sense that $x \mapsto U_0(x) - \frac{\lambda^*}{2} \|x \|_L^2$ is a convex function.  
\end{itemize}
We assume that $\lambda^*>0$ and refer to $\kappa^*=1/\lambda^*$ as coordinate-wise condition number of $U$.
\end{asmp}

In Assumption \ref{asmp:convex_smooth_block} the separable part $\sum_m U_m(x_m)$ needs not to be smooth nor strongly convex, only convex. 
The $\lambda^*$-convexity of $U_0$ with respect to $\|\cdot\|_L$ is equivalent to the usual $\lambda^*$-convexity of the function $\hat{U}_0 (x)= U_0(x_1/\sqrt{L_1}, \ldots, x_M / \sqrt{L_M})$.
Thus we can check part (ii) of Assumption~\ref{asmp:convex_smooth_block} proceeding in two steps: first rescale each coordinate $x_m$ by $L_m^{-1/2}$ so that its partial derivative becomes $1$-Lipschitz, secondly check what is the convexity constant $\lambda^*$ of the resulting potential.

Assumption~\ref{asmp:convex_smooth_block} is strictly weaker than Assumption~\ref{asmp:convex_smooth}, as stated in the following.
\begin{lemma}
If $U$ satisfies Assumption~\ref{asmp:convex_smooth} then it satisfies Assumption~\ref{asmp:convex_smooth_block} and $1 \leq \kappa^* \leq \kappa$.
\end{lemma}

\begin{proof}
Let $U$ satisfy Assumption~\ref{asmp:convex_smooth},
which implies that it is block-smooth with $L_m \leq L$ for all $m$.
Thus, taking $U_m = 0$ for $m=1, \ldots, M$ and $U_0 = U$, we have $\| \cdot \|_L \leq L \| \cdot \|$ and thus $U$ satisfies also Assumption \ref{asmp:convex_smooth_block} with $\lambda^* \geq \lambda/L$. 
\end{proof}

To give an intuition about the difference between $\kappa$ and $\kappa^*$, we mention the characterization of the extremal cases $\kappa=1$ and $\kappa^*=1$. Assume $d_1=\dots=d_M=1$ so that $M=d$. Then it is easy to prove the following.
\begin{enumerate}
\item The function $U$ satisfies Assumption~\ref{asmp:convex_smooth} with $\kappa=1$ if and only if $U$ is quadratic and isotropic, that is, $U(x) = \lambda \| x \|^2$ for some $\lambda > 0$.
\item The function $U$ satisfies Assumption~\ref{asmp:convex_smooth_block} with $\kappa^*=1$ if and only if $U$ is separable, that is, $U(x) = \sum_{m=1}^d U_m(x_m)$, and it is $\lambda$-convex for some $\lambda > 0$.
This follows by taking $U_0(x)=\lambda\|x\|^2/2$.
\end{enumerate}
The second case corresponds to $\pi$ which factorizes across blocks, which is when the GS converges the fastest (Lemma \ref{lm:lower_bound_KL}).

\begin{remark}[Checking the assumptions for $C^2$ potentials]\label{rmk:C2_asmp}
If $U$ is of class $C^2$, writing $\nabla^2 U$ for its Hessian and $A \leq B$ for symmetric matrices $A,B$ if $B-A$ has only non-negative eigenvalues, it is well known that Assumption~\ref{asmp:convex_smooth} corresponds to 
\begin{equation*}
\lambda \Id_d \leq \nabla^2 U(x) \leq L \Id_d\,,
\end{equation*}
for all $x \in \R^d$, where $\Id_d$ denotes the $d\times d$ identity matrix. 
On the other hand, if $U_0$ is of class $C^2$, part (i) of Assumption~\ref{asmp:convex_smooth_block} is equivalent to 
  \begin{equation*}
[\nabla^2 U_0(x)]_{mm} \leq L_m \Id_{d_m}\,,
\end{equation*}
for all $x\in \R^d$ and $m=1, \ldots, M$, where $[\nabla^2 U_0(x)]_{mm}$ denotes the $d_m \times d_m$ diagonal block of the matrix $\nabla^2 U_0(x)$;
 while part (ii) becomes equivalent to
 \begin{equation*} 
\lambda^* D \leq \nabla^2 U_0(x),
\end{equation*}
for all $x\in \R^d$, where $D$ denotes the diagonal matrix with diagonal coefficients $L_1, L_2, \ldots, L_M$, with each $L_m$ repeated $d_m$ times, that is, $D_{mm} = L_m \Id_{d_m}$.
Equivalently, part (ii) of Assumption~\ref{asmp:convex_smooth_block} 
imposes 
 $\lambda^*\Id_d\leq D^{-1/2}\nabla^2 U_0(x) D^{-1/2}$,
uniformly in $x \in \R^d$. 
\end{remark}

\begin{remark}[Separable potential with a quadratic interaction term]\label{rmk:sep_pot_quadr}
Following \cite{yuan2023class}, assume $M=2$ and $U(x_1,x_2)=V_1(x_1)+V_2(x_2)+\|x_1-x_2\|^2/(2h)$ for some $h>0$, with $V_1$ and $V_2$ strongly convex with parameters $\lambda_1$ and $\lambda_2$, respectively. Then, with the help of Remark \ref{rmk:C2_asmp}, we can check that $U$ satisfies Assumption \ref{asmp:convex_smooth_block} with  $U_0(x)=\|x_1-x_2\|^2/(2h)+\sum_{m=1}^2\lambda_m\|x_m\|^2/2$, 
$U_m(x_m)=V_m(x_m)-\lambda_m\|x_m\|^2/2$ for $m\in\{1,2\}$ and $\lambda^*=1-((1+h\lambda_1)(1+h\lambda_2))^{-1/2}$.
We will return to this example in Remark \ref{rmk:prox_sampl} below.
\end{remark}

\section{Main result: entropy contraction of the Gibbs sampler}
\label{sec:main}

We start with the analytical form of our main result: an inequality connecting marginal and joint entropies, which can be of independent interest. 

\begin{theorem}
\label{theo:functional_main}
Let $\pi$ satisfy Assumption~\ref{asmp:convex_smooth_block}.  
Then for any $\mu \in \sP(\R^d)$
\begin{equation}\label{eq:functional_main}
\frac{1}{M} \sum_{m=1}^M \KL(\mu_{-m} | \pi_{-m}) \leq \left( 1 - \frac{1}{\kappa^* M} \right) \KL(\mu | \pi). 
\end{equation}
\end{theorem}

From the chain rule~\eqref{eq:KL_chain_rule}, this inequality is equivalent to  
\begin{equation}
\label{eq:functional_main_alternative}
\sum_{m=1}^M  \E_{X_{-m} \sim \mu_{-m}} \left( \KL(\mu(\cdot | X_{-m} ) | \pi(\cdot | X_{-m} )) \right) \geq \frac{1}{\kappa^*} \KL(\mu|\pi)\,,
\end{equation}
which is also referred to as approximate tensorization of the entropy in  the literature, see e.g.\ \cite{caputo2023lecture} and references therein. 
This result is our main technical contribution. Its proof is postponed until Section~\ref{sec:proof_main}, whereas we now focus on the many implications of this inequality. We first show that it directly implies entropy contraction of GS.

\begin{theorem}
\label{theo:contraction_KL_main}
Let $\pi$ satisfy Assumption~\ref{asmp:convex_smooth_block}. Then for any $\mu \in \sP(\R^d)$
\begin{equation}\label{eq:contraction_KL_main}
\KL( \mu P^\mathrm{GS} | \pi ) \leq \left( 1 - \frac{1}{\kappa^* M} \right) \KL(\mu | \pi). 
\end{equation}
\end{theorem}

\begin{proof}
Using the definition of $P^{GS}$, the convexity of $\KL$ and the variational characterization of the kernels $P_m$ (Lemma~\ref{lm:variational_characterization_Pm}), we have
\begin{equation}\label{eq:funct_to_contr}
\KL( \mu P^\mathrm{GS} | \pi ) = \KL \left( \left. \frac{1}{M} \sum_{m=1}^M \mu P_m \right| \pi \right)  \leq \frac{1}{M} \sum_{m=1}^M \KL(\mu P_m | \pi) = \frac{1}{M} \sum_{m=1}^M \KL(\mu_{-m} | \pi_{-m}).    
\end{equation}
Then we conclude with Theorem~\ref{theo:functional_main}. 
\end{proof}
While closely related, \eqref{eq:functional_main} is stronger than \eqref{eq:contraction_KL_main}, since the convexity inequality in \eqref{eq:funct_to_contr} could be strict. With the vocabulary of \cite{caputo2025entropy}, the equation~\eqref{eq:functional_main} is the half-step contraction of the Markov chain, while~\eqref{eq:contraction_KL_main} is the full-step contraction, and at least for GS on discrete spaces the two may significantly differ \cite[Section 3.2]{caputo2025entropy}. As we will see below in Lemma~\ref{lm:lower_gaus}, on Gaussian targets both rates are tight up to constant.
In Section \ref{sec:non_strongly_convex} we will show how Theorem \ref{theo:contraction_KL_main} can be adapted to the non-strongly convex case, when $\lambda^*=0$. 

\begin{remark}
Regarding Remark~\ref{rmk:idp_parametrization}, note that our results are at least invariant with respect to \emph{isotropic linear} rescalings of coordinates. That is, take $t_1, \ldots, t_M \in \R$ strictly positive, let $\phi : (x_1, \ldots, x_M) \to (t_1 x_1, \ldots, t_M x_M)$, and define $\hat{\pi} = \phi_\# \pi$. Then we leave the reader to check that $\kappa^*$ is the same for $\pi$ and $\hat{\pi}$, so that the entropy contraction coefficient of Theorem~\ref{theo:contraction_KL_main} is the same for $\pi$ and $\hat{\pi}$. Similarly, it can be checked that, in Theorem~\ref{theo:non_strongly_convex} below, the quantity $R$ does not change if both $\pi$ and $\mu^{(0)}$ are rescaled by $\phi$. However it does not exclude that, for a given $\pi$, one may find a (possibly non-linear) transformation $\phi$ preserving the block structure, so that the upper bound given by a result such as Theorem~\ref{theo:contraction_KL_main} is much tighter for $\hat{\pi} = \phi_\# \pi$ than for $\pi$.
\end{remark}

\begin{remark}[Comparison with previous results on the proximal sampler]\label{rmk:prox_sampl}
\cite{yuan2023class} provides entropy contraction rates for potentials $U$ as in Remark \ref{rmk:sep_pot_quadr}. In that context,  
Assumption \ref{asmp:convex_smooth_block} holds with $\lambda^*=1-((1+h\lambda_1)(1+h\lambda_2))^{-1/2}$, and Theorems \ref{theo:functional_main}-\ref{theo:contraction_KL_main} provide results that are equivalent to \cite[Thm. 5]{yuan2023class} up to constants. 
The results in \cite{yuan2023class} are stated for deterministic-scan GS, while Theorem \ref{theo:contraction_KL_main} deals with the random-scan variant: nonetheless, when $M=2$ the two are closely related and the approximate tensorization of the entropy in Theorem \ref{theo:functional_main} is enough to control both variants, see e.g.\ \cite[Thm. 7.3]{ascolani2025mixing}.
More generally, the convergence of GS with $M=2$ and quadratic interaction terms has been previously studied by multiple authors in the context of Gaussian localization schemes \cite[Thm. 58]{chen2025localization} or proximal samplers \cite[Ch. 8]{chewi2023log}.
Such case is particularly tractable and amenable to various types of analyses, including direct Wasserstein contraction arguments, see Chapter 8 of \cite{chewi2023log} and references therein for an overview of available results.
\end{remark}

\subsection{Implications for mixing time}\label{sec:mix_times}
Uniform entropy contraction results, such as Theorem \ref{theo:contraction_KL_main}, provide strong control over the convergence speed of the associated Markov chain, as we now briefly discuss. 
First, recalling that $\mu^{(n+1)}=\mu^{(n)}P^\mathrm{GS}$, iterating \eqref{eq:contraction_KL_main} implies that
\begin{equation}\label{eq:basic_KL_mix_time_bound}
n\geq\kappa^*  M\log\left(\frac{\KL(\mu^{(0)}|\pi)}{\epsilon} \right)    
\end{equation}
is enough to ensure $\KL(\mu^{(n)}|\pi)\leq \epsilon$.
Thus to have a fully explicit mixing time, we need to bound $\KL(\mu^{(0)}|\pi)$. 
We report the resulting bounds in case of warm and feasible starts.
As feasible start, we consider two alternatives: the Gaussian one usually considered in the literature on log-concave sampling, and a factorized one that is natural in the GS context.

\begin{corollary}[Mixing time bounds - KL case]
\label{crl:GS_mix_time}
Let $\epsilon>0$.
\begin{enumerate}
    \item[(a)] \emph{(Warm start)}
If $\pi$ satisfies Assumption \ref{asmp:convex_smooth_block} and $\mu^{(0)}\in\sP(\R^d)$ is such that $\mu^{(0)}(A)\leq C\pi(A) \hbox{ for all }A\in \mathcal{B}(\R^d)$, we have 
$\KL(\mu^{(n)}|\pi) \leq \epsilon$ whenever
$$
n\geq\kappa^*  M\left[\log\left(\frac{1}{\epsilon} \right)+\log(\log(C)) \right]\,.
$$
    \item[(b)]  \emph{(Feasible start: Gaussian)} If $\pi$ satisfies Assumption \ref{asmp:convex_smooth}, $x^*$ is the unique mode of $\pi$, and $\mu^{(0)}=N(x^*, L^{-1}\Id_d)$ is a Gaussian distribution centered at $x^*$ and with covariance matrix $L^{-1}\Id_d$, we have 
$\KL(\mu^{(n)}|\pi) \leq \epsilon$ whenever
$$
n\geq\kappa  M\left[\log\left(\frac{1}{\epsilon} \right)+ \log \left( \frac{d}{2} \right) + \log (\log(\kappa)) \right]\,.
$$
    \item[(c)]  \emph{(Feasible start: factorized)} If $\pi$ satisfies Assumption \ref{asmp:convex_smooth}, $x^*$ is the unique mode of $\pi$, and
    $ \mu^{(0)}(\d x) = \bigotimes_{m=1}^M \pi(\d x_m|x^*_{-m} )$, we have 
$\KL(\mu^{(n)}|\pi) \leq \epsilon$ whenever
$$
n\geq\kappa  M\left[\log\left(\frac{1}{\epsilon} \right)+ \log(d) + 2 \log(\kappa) \right]\,.
$$
\end{enumerate}
\end{corollary}
\begin{proof}
Part (a) follows from \eqref{eq:basic_KL_mix_time_bound} and $\KL(\mu^{(0)}|\pi)\leq\sup_{x} \log \left( \frac{\d \mu^{(0)}}{\d \pi}(x) \right)\leq \log(C)$. Part (b) follows from part (a) and the fact that, under Assumption \ref{asmp:convex_smooth} and $\mu^{(0)}=N(x^*, L^{-1}\Id_d)$, one has $\mu^{(0)}(A)\leq C\pi(A)$ for all $A\in\mathcal{B}(\R^d)$ with $C=\kappa^{d/2}$, see e.g.\ \cite[Equation (12)]{D19}. Part (c) follows from~\eqref{eq:basic_KL_mix_time_bound} and the bound $\KL(\mu^{(0)}|\pi) \leq d \kappa^2$ which we prove in the \hyperref[appn]{Appendix}.
\end{proof}

The existence of warm starts is often assumed in the sampling literature, but finding one with C that does not grow exponentially with $d$ can be hard.
Thus, the double-logarithmic dependence on $C$ is convenient to avoid additional terms growing polynomially with $d$ in the case of a Gaussian feasible start. Whereas the Gaussian feasible start requires the knowledge of $L$, implementing the factorized feasible start requires the same ingredients as implementing the Gibbs sampler: after finding the mode $x^*$, one samples independently $X^{(0)}_m \sim \pi(\cdot | x^*_{-m})$ for $m=1, \ldots, M$. It results in a start which is in general not warm, but still enjoys a logarithmic scaling of the mixing time in the dimension.   

\begin{remark}[Extension to TV and $W_2$]\label{rmk:mix_TV_KL}
Thanks to classical functional inequalities, mixing times in KL can be directly extended to other metrics. Pinsker's inequality (see e.g.~\cite[Eq. (22.25)]{villani2009optimal}) relates the total variation ($\mathrm{TV}$) distance $\| \mu-\pi \|_\mathrm{TV}$ to $\KL$ via $\| \mu-\pi \|_\mathrm{TV}^2\leq \KL(\mu|\pi)/2$. Talagrand's inequality relates the Wasserstein distance $W_2(\mu,\pi)$ to $\KL$ via $W_2^2(\mu,\pi) \leq 2 \KL(\mu| \pi)/\lambda$ if $U$ is $\lambda$-convex with $\lambda>0$. Thus the mixing times of $P^{\mathrm{GS}}$ with respect to $\mathrm{TV}$ and $W_2$ have the same scaling as the ones in Corollary \ref{crl:GS_mix_time}, up to multiplicative constants.   
\end{remark}

\subsection{Related functional inequalities}
\label{sec:spectral_gap}

Theorem \ref{theo:functional_main} also implies a lower bound on the spectral gap of $P^\mathrm{GS}$. The spectral gap of a $\pi$-reversible kernel $P$ is defined as
\begin{equation}\label{eq:gap}
\gap(P) = \inf_{f \in L^2(\pi)} \frac{\langle f, (\Id-P)f \rangle_{\pi}}{\mathrm{Var}_\pi(f)},   
\end{equation}
where $\langle f, g \rangle_{\pi} = \int f(x)g(x)\pi(\d x)$ and $\mathrm{Var}_\pi(f) = \int f^2(x) \pi(\d x) - \left(\int f(x) \pi(\d x) \right)^2$, with $f,g \in L^2(\pi)$. 
We then have the following result.
\begin{corollary}
\label{crl:variance_inequality}
Let $\pi$ satisfy Assumption~\ref{asmp:convex_smooth_block}.
Then $\gap(P^\mathrm{GS}) \geq 1/(\kappa^* M)$. Equivalently, if $X \sim \pi$ and $f \in L^2(\pi)$ it holds
\begin{equation}\label{eq:variance_inequality} 
\sum_{m=1}^M \E( \mathrm{Var}(f(X) | X_{-m} ) ) \geq \frac{1}{\kappa^*}
\mathrm{Var}(f(X)).
\end{equation}
\end{corollary}
\begin{proof}
The bound~\eqref{eq:variance_inequality} follows directly from the approximate tensorization of the entropy in~\eqref{eq:functional_main_alternative} by linearization, specifically using $\mu = (1+ \varepsilon \tilde{f}) \cdot \pi $ with $\varepsilon \to 0$, for $\tilde{f} = f - \E(f(X))$. The argument is well-known, see e.g. \cite[Prop. 1.1]{caputo2015approximate}. Once we have~\eqref{eq:variance_inequality} the result on the gap follows as
Lemma~\ref{lm:interpretation_Pm} implies $\langle f, (\Id-P_m) f \rangle_{\pi} = \E( \mathrm{Var}(f(X) | X_{-m} ) )$.
Thus, by the definition of $P^\mathrm{GS}$ in \eqref{eq:def_GS}, we obtain
\begin{equation}
\label{eq:expression_Dir_PGS}
\langle f, (\Id-P^\mathrm{GS}) f \rangle_{\pi} = \frac{1}{M} \sum_{m=1}^M \E( \mathrm{Var}(f(X) | X_{_m} ) ),
\end{equation}
so that $\gap(P^\mathrm{GS}) \geq 1/( \kappa^* M)$ is equivalent to~\eqref{eq:variance_inequality}. 
\end{proof}
The inequality \eqref{eq:variance_inequality} provides a lower bound on the conditional variances of functions in $L^2(\pi)$ in terms of their marginal variance.
To the best of our knowledge, it did not appear previously in the literature, and we have not been able to provide a direct proof of it using classical functional inequalities without passing through relative entropy.

In addition it is well understood that~\eqref{eq:functional_main} is equivalent to a Brascamp-Lieb inequality for the measure $\pi$. In our case it takes the following form.
\begin{corollary}
Let $\pi$ satisfy Assumption~\ref{asmp:convex_smooth_block}, and let $f_1, \ldots, f_M$ measurable non-negative functions defined on $\sX_{-1}, \ldots, \sX_{-M}$ respectively. Then, with $\theta = M - 1/\kappa^*$, 
\begin{equation}
\label{eq:BL}
\E_\pi \left[ \prod_{m=1}^M f_m(X_{-m}) \right] \leq \prod_{m=1}^M \E_\pi\left[ f_m(X_{-m})^{\theta} \right]^{1/\theta}.
\end{equation}
\end{corollary}
\begin{proof}
The inequality~\eqref{eq:BL} is equivalent to~\eqref{eq:functional_main} and follows by Legendre duality, see e.g. \cite[Theorem 2.1]{carlen2009subadditivity} or \cite[Theorem 1]{caputo2024entropy}.
\end{proof}

\subsection{Tightness of the entropy contraction rate}
\label{subsec:tightness}

We conclude this section by discussing the tightness of our rate.
Out of clarity we introduce some constants capturing the worst possible case. For any $\pi$-reversible kernel $P$ we define the entropy contraction constant as
\begin{equation*}
\rho_\mathrm{EC}(P):=\sup_{\mu\in\sP(\R^d)}\frac{\KL( \mu P | \pi ) }{\KL( \mu| \pi ) },
\end{equation*}
where the supremum is restricted to measures $\mu\in\sP(\R^d)$ such that $\KL( \mu | \pi )<\infty$. We also write $C_\mathrm{ATE}(\pi)$ for the optimal constant in the approximate tensorization of the entropy, that is the smallest $C \geq 0$ such that 
\begin{equation*}
C\,\sum_{m=1}^M  \E_{X_{-m} \sim \mu_{-m}} \left( \KL(\mu(\cdot | X_{-m} ) | \pi(\cdot | X_{-m} )) \right) \geq \KL(\mu|\pi),
\end{equation*}
holds for all measures $\mu \in \sP(\R^d)$ with $\KL( \mu | \pi )<\infty$. Written like this,~\eqref{eq:functional_main_alternative} is equivalent to $C_\mathrm{ATE}(\pi) \leq \kappa^*$. Recalling that the spectral gap is defined in~\eqref{eq:gap}, we have, in the context of the Gibbs sampling, the following comparisons:
\begin{align}
\label{eq:link_EC_ATE}
\rho_\mathrm{EC}(P^\mathrm{GS}) & \leq 1 - \frac{1}{M C_\mathrm{ATE}(\pi)}, \\
\label{eq:link_G_ATE}
\gap(P^\mathrm{GS}) & \geq \frac{1}{M C_\mathrm{ATE}(\pi)},
\end{align}
the first one being a rewriting of the proof of Theorem~\ref{theo:contraction_KL_main} from Theorem~\ref{theo:functional_main}, while the second one follows from the proof of Corollary~\ref{crl:variance_inequality}. Thus the estimate $C_\mathrm{ATE}(\pi) \leq \kappa^*$ from Theorem~\ref{theo:functional_main} implies $\rho_\mathrm{EC}(P^\mathrm{GS}) \leq 1 - 1/(\kappa^* M)$ and $\gap(P^\mathrm{GS}) \geq 1/(\kappa^* M)$ as we already stated. 

We also recall the following well-known connection between entropy contraction coefficients and spectral gaps which follows by linearization (e.g. \cite{caputo2023lecture}, below Lemma 2.15): for any $\pi$-reversible kernel $P$ we have
\begin{equation}\label{eq:gap_and_entropies}
\rho_\mathrm{EC}(P) \geq 1-2\gap(P)\,.
\end{equation}

We first recall the following result: as one iteration of GS changes only a proportion $1/M$ of the coordinates, it yields universal bounds which hold for any target distribution, even without log-concavity. 

\begin{lemma}
\label{lm:lower_bound_KL}
For any $\pi \in \sP(\R^d)$, we have 
\begin{equation*}
\gap(P^\mathrm{GS}) \leq \frac{1}{M}, \qquad C_\mathrm{ATE}(\pi) \geq 1, \qquad \rho_{\mathrm{EC}}(P^\mathrm{GS}) \geq 1 - \frac{2}{M};
\end{equation*}
and the first two inequalities are equalities if and only if $\pi$ factorizes over $\sX_1 \times \ldots \times \sX_M$.
\end{lemma}

\begin{proof}
For $C_\mathrm{ATE}$, let $\mu\in\sP(\R^d)$ admit a density $f$ with respect to $\pi$, with $f$ depending only on one coordinate $x_m \in \sX_m$. 
In this case we have  for $m' \neq m$
\[
\KL(\mu \mid \pi) = \int_{\sX_m}f(x_m) \log(f(x_m))\pi_m(\d x_m) = \KL(\mu_{-m'} \mid \pi_{-m'})\,,
\]
where $\pi_m$ denotes the marginal distribution of $X_m$ under $X\sim\pi$. 
Thus by the chain rule~\eqref{eq:KL_chain_rule}, we obtain 
\begin{align*}
\E_{X_{-m'} \sim \mu_{-m'}} \left( \KL(\mu(\cdot | X_{-m'} ) | \pi(\cdot | X_{-m'} )) \right) & =  \KL(\mu \mid \pi)- \KL(\mu_{-m'} \mid \pi_{-m'}) \\
& =
\begin{cases}
0 & \text{if } m' \neq m, \\
\leq \KL(\mu | \pi) & \text{if } m = m'. 
\end{cases} 
\end{align*}
Summing these inequalities for $m' = 1, \ldots, M$, we see that necessarily $C_\mathrm{ATE}(\pi) \geq 1$. Moreover, if there is the equality $C_\mathrm{ATE}(\pi) = 1$ then necessarily $\KL(\mu_{-m} | \pi_{-m}) = 0$ for $\mu \in \sP(\R^d)$ having a density $f$ depending only on $x_m \in \sX_m$, meaning $\mu_{-m} = \pi_{-m}$ in this case. As the density of $\mu_{-m}$ with respect to $\pi_{-m}$ is $P_m f$, it implies $P_m f = 1$ for all $f = f(x_m)$ non-negative which integrate to $1$. By linearity it implies $\E_\pi[f(X_m)|X_{-m}] = \E_\pi[f(X_m)]$ a.s. for all $f$ measurable and bounded, which implies that $X_m$ is independent from $X_{-m}$. As $m$ is arbitrary in this reasoning, it implies that $\pi$ factorizes. 

For the gap, using also a test function $f$ depending only on $x_m$ in the definition~\eqref{eq:gap}, together with~\eqref{eq:expression_Dir_PGS}, provides $\gap(P^\mathrm{GS}) \leq 1/M$. If $\gap(P^\mathrm{GS}) = 1/M$, then with the same test function we see that $\E( \mathrm{Var}(f(X_m) | X_{-m} ) ) = 0$ for all $m$ and all bounded function $f$ defined on $\sX_m$ with zero mean, which again implies that $\pi$ factorizes.

The inequality on $\rho_\mathrm{EC}(P^\mathrm{GS}) \geq 1-2/M$ follows from the one of the gap and~\eqref{eq:gap_and_entropies}.

Eventually, if $\pi$ factorizes over $\sX_1 \times \ldots \times \sX_M$, then the so-called tensorization of the entropy (see e.g.~\cite[Proposition 1.4.1]{ane2000inegalites} or~\cite[Lemma 22.8]{villani2009optimal}) precisely states that $C_\mathrm{ATE}(\pi) = 1$. The equality $\gap(P^\mathrm{GS}) = 1/M$ follows by $\gap(P^\mathrm{GS}) \leq 1/M$ and~\eqref{eq:link_G_ATE}.
\end{proof}

To illustrate the tightness of our bounds, we can consider the Gaussian case. Indeed, in this case, the following lemma implies that Theorem~\ref{theo:functional_main} is tight, and the upper bound $\rho_\mathrm{EC}(P^\mathrm{GS})\leq 1-1/(d\kappa^*)$ of Theorem \ref{theo:contraction_KL_main} is tight, up to a multiplicative factor of $2$. 
\begin{lemma}\label{lm:lower_gaus}
    Let $M = d$ and $\pi = N(x^*,Q^{-1})$ be a Gaussian distribution with precision matrix $Q$. 
   Then $\kappa^*=1/\lambda_\mathrm{min}(D^{-1/2}QD^{-1/2})$, where $D$ denotes the diagonal matrix with diagonal terms equal to $Q$.
   Moreover, 
    \begin{equation*}
    \gap(P^\mathrm{GS}) = \frac{1}{d \kappa^*}, \qquad C_\mathrm{ATE}(\pi) = \kappa^*, \qquad \rho_\mathrm{EC}(P^\mathrm{GS}) \in \left[ 1 - \frac{2}{\kappa^* d}, 1 - \frac{1}{\kappa^* d} \right].
    \end{equation*}
\end{lemma}
\begin{proof}
The expression of $\kappa^*$ follows from Remark \ref{rmk:C2_asmp}. The computation of the spectral gap can be found in \cite[Theorem 1]{amit1996convergence}. As in our case $d= M$, then $C_\mathrm{ATE}(\pi) \leq \kappa^*$ by~\eqref{eq:functional_main_alternative} while $C_\mathrm{ATE}(\pi) \geq \kappa^*$ follows from~\eqref{eq:link_G_ATE}.
The upper bound for $\rho_\mathrm{EC}(P^\mathrm{GS})$ is Theorem~\ref{theo:contraction_KL_main}, while the lower bound follows from~\eqref{eq:gap_and_entropies}.
\end{proof}

The recent work \cite[Theorem 3]{caputo2024entropy} also derives the explicit result $C_\mathrm{ATE}(\pi) = \kappa^*$ when $U = x^\top (\mathrm{Id} - S) x /2$, with $S$ a symmetric substochastic matrix. Our results thus extend theirs to the whole class of Gaussian measures.

\section{Computational considerations}
\label{section:computational_cost}

In this section we first discuss how to implement GS for log-concave distributions, and then move to the main computational implications of Theorem \ref{theo:contraction_KL_main} and connections with the optimization literature. 
Even if our results apply to the more general Assumption \ref{asmp:convex_smooth_block}, here we mostly focus on Assumption \ref{asmp:convex_smooth} to simplify the discussion and to facilitate the comparison with other MCMC schemes.

\subsection{Implementation of GS}\label{sec:implementation}

In order to implement Algorithm \ref{alg:rs-gs}, exact sampling from the full conditionals of $\pi$, i.e.\ $\pi(\cdot|x_{-m})$ for an arbitrary $x_{-m}\in\sX_{-m}$, is required. 
Here we focus on the case where $\pi$ is a generic strongly log-concave distribution satisfying Assumption \ref{asmp:convex_smooth}, even if we note that in many applications $\pi(\cdot|x_{-m})$ belongs to some known class of analytically tractable distributions (i.e.\ the model is ``conditionally-conjugate'' \cite{C92}), so that sampling from $\pi(\cdot|x_{-m})$ can be accomplished with distribution specific algorithms. 
If $\pi$ satisfies Assumption \ref{asmp:convex_smooth}, then so does $\pi(\cdot|x_{-m})$ with the same condition number $\kappa$, since log-concavity is preserved under conditioning.
Sampling from low-dimensional log-concave distributions can be accomplished with (adaptive) Rejection Sampling (RS) algorithms \citep{gilks1992adaptive}. 
This is, for example, the default option implemented in generic software implementations of GS \citep{lunn2009bugs,JAGS}. 
In particular, \cite{chewi2022query} has recently analyzed RS algorithms for one-dimensional log-concave distributions (i.e. $d_m = 1$ in our context) showing that, on average, they only require $\sO(\log(\kappa))$ evaluations of the conditional density (or $\sO(\log(\log(\kappa)))$ if one has access to the conditional mode) for each sample from $\pi(\cdot|x_{-m})$.  
This is in agreement with empirical observations reporting that adaptive RS usually requires only a few (e.g.\ less than $5$) full conditional evaluations to produce a sample from $\pi(\cdot|x_{-m})$ \citep[Table 1]{gilks1992adaptive}.
Thus, in the following we assume that, when $d_m=1$, exact draws from $\pi(\cdot|x_{-m})$ can be obtained with $\sO(\log (\kappa))$ evaluations of the associated full conditional density.

\subsection{Comparison with gradient-based MCMC schemes}\label{sec:comparison_gradient}

To compare GS with alternative MCMC schemes, most notably gradient-based ones, we work under the assumption that evaluating a single full conditional is $M$-times cheaper than evaluating the full likelihood. Intuitively, it captures the class of problems that are computationally amenable to coordinate-wise methods such as GS. 

\begin{asmp}\label{ass:comp_cost}
Given two arbitrary points $x, y \in \R^d$, the computational cost of evaluating $U(y_m,x_{-m})-U(x)$ (averaged over $m=1,\dots,M$) is $M$-times smaller than the one of computing $U(y)-U(x)$ or $\nabla U(x)$.
\end{asmp}

In the optimization literature on coordinate descent methods, Assumption \ref{ass:comp_cost} is usually equivalently stated by saying that computing a partial derivative of $U$ is $M$-times cheaper than evaluating the full gradient $\nabla U$. This is an intuitive property that holds for many, but not all, models (see e.g.\ \cite{nesterov2012efficiency, RT14,W15} for detailed discussions and examples). 
As a common and relevant example, we mention the case where
\begin{equation}\label{def:simple_U}
    U(x) = \sum_{i = 1}^df_i(x_i)+f(Ax),
\end{equation}
where $A \in \mathbb{R}^{r\times d}$ for some integer $r>0$ and $f, f_i$, with $i = 1,\dots, d$, are smooth functions (which require $\sO(1)$ cost to be evaluated for a given input). By keeping track of $Ax$ during iterations, it is clear that computing a new value of $U$ after changing only $d_m$ coordinates, i.e.\ computing $U(y_m,x_{-m})-U(x)$ with $y_m\in\sX_m$, requires $\sO(rd_m)$ operations while computing  $U(y)-U(x)$ for a general $y\in\R^d$ requires $\sO(rd)$ operations.
Thus Assumption \ref{ass:comp_cost} is satisfied under \eqref{def:simple_U}. 
For example, target distributions as in \eqref{def:simple_U} arise from generalised linear models (e.g.\ linear or logistic regression) with independent priors on the regression coefficients $x$ and $r$ observations (see e.g.\ \cite[Ch.16]{gelman2013bayesian}  and the recent work \cite{Luu25}, which also provides implementations where Assumption~\ref{ass:comp_cost} holds for some generalised linear models).

Taking $d_m=1$ for all $m$ for simplicity, Corollary \ref{crl:GS_mix_time}, Remark \ref{rmk:mix_TV_KL} and the discussion in Section \ref{sec:implementation} imply that, ignoring logarithmic terms in $\kappa$ and $d$, GS requires $\sO\left(\kappa d\log(1/\epsilon)\right)$ full-conditional evaluations 
to obtain a sample from the target $\pi$ up to an error $\epsilon$ in TV (or some other) distance, both under a warm and a feasible start. 
Thus, under Assumptions \ref{asmp:convex_smooth} and \ref{ass:comp_cost}, GS produces an $\epsilon$-accurate sample with a computational cost of order
$$
\sO\left(\kappa \log(1/\epsilon)\right)
$$
evaluations of the target $U$ (or equivalently $\nabla U$). 
Note that here we are using the bound $\kappa^*\leq \kappa$ for simplicity, ignoring the fact that $\kappa^*$ could be potentially much smaller than $\kappa$. In Table~\ref{table:computational_cost} we compare this cost with the one of three alternative MCMC methods: Random Walk Metropolis (RWM), 
the Metropolis-Adjusted-Langevin Algorithm (MALA)
and Hamiltonian Monte Carlo (HMC), using results from \cite{AL22} for RWM, \cite{wu2022minimax} for MALA and \cite{chen2023does,chen2020fast} for HMC.
In the interest of brevity, we do not include unadjusted MCMC schemes (such as unadjusted overdamped and underdamped Langevin algorithms, see e.g.\ \cite{chewi2023log} and references therein) in the comparison; we just mention that the considerations below apply also to those schemes.
\begin{table}[h!]
\caption{Available upper bounds on the computational cost (measured in terms of total number of target evaluations, either of $U$ or $\nabla U$) required to sample from a $d$-dimensional strongly log-concave distribution with condition number $\kappa$ for different MCMC schemes, ignoring logarithmic terms and the dependence on $\epsilon$. For GS, we employ Assumption \ref{ass:comp_cost}. 
For HMC, the warm start bound requires additional regularity assumptions on $\pi$, see \cite{chen2023does}.
Warm and feasible starts are defined as in part (a) and (b) of Corollary \ref{crl:GS_mix_time} above. References for all the entries of the table are given in Section \ref{sec:comparison_gradient}.}
\begin{tabular}{||c c c ||} 
 \hline
 Method & Warm start & Feasible start \rule{0pt}{2.6ex} \\ [0.5ex] 
 \hline\hline
 RWM & $\sO(\kappa d)$ & $\sO(\kappa d)$ \rule{0pt}{2.6ex} \rule[-0.9ex]{0pt}{0pt} \\
 \hline 
 MALA & $\sO(\kappa \sqrt{d})$  & $\sO(\kappa d)$ \rule{0pt}{2.6ex} \rule[-0.9ex]{0pt}{0pt} \\
 \hline
 HMC & $\sO(\kappa d^{\frac{1}{4}})$ & $\sO(\kappa d^{\frac{11}{12}})$ \rule{0pt}{2.6ex} \rule[-0.9ex]{0pt}{0pt}  \\
 \hline
 GS & $\sO(\kappa)$ & $\sO(\kappa)$ \rule{0pt}{2.6ex} \rule[-0.9ex]{0pt}{0pt} \\  
 \hline
\end{tabular}
\label{table:computational_cost}
\end{table}

According to the results reported in Table \ref{table:computational_cost}, GS outperforms competing schemes therein. 
In particular, to the best of our knowledge, GS is the only scheme that, under Assumptions \ref{asmp:convex_smooth} and \ref{ass:comp_cost}, produces an $\epsilon$-accurate sample with a dimension-free number of target evaluations. 
On the other hand, the linear dependence on $\kappa$ is tight for GS even in the Gaussian case (see Lemma~\ref{lm:lower_gaus}), while it is sub-optimal for HMC in that case: indeed \cite{apers2022hamiltonian} prove that, in the specific case of Gaussian targets, HMC requires only $\sO\left(\sqrt{\kappa}d^{1/4}\right)$ gradient evaluations for each sample. 
While more research is needed to get a full picture, these results are in agreement with empirical observations that report HMC being more robust than GS to bad conditioning (arising from e.g.\ bad parametrizations with strongly correlated parameters) and, on the contrary, GS being more efficient in high-dimensional scenarios with good parametrizations \cite{papaspiliopoulos2023scalable,Luu25}.

\subsection{Comparison with optimization: similarities and differences}\label{sec:comparison_optimization}

Coordinate descent methods, which minimize a (convex) function $U$ on $\R^d$ by iteratively optimizing it with respect to different subsets of variables, received considerable attention in the last decade, starting from the pioneering paper by \cite{nesterov2012efficiency} with several subsequent works (see e.g.\ \cite{beck2013convergence, RT14,W15}). They are known to be state-of-the-art methods for some classes of problems, both in theory and practice (e.g. fitting generalized linear models with Lasso penalty, \cite{friedman2010regularization}). 
More broadly, coordinate-wise methods are competitive when computing a partial derivative of $U$ is much cheaper than computing its full gradient (i.e.\  Assumption \ref{ass:comp_cost} above). Within this perspective, the results of this paper provide a close analogue to the ones obtained in the optimization literature: indeed the rate we prove in Theorem \ref{theo:contraction_KL_main} exactly matches the ones available for random-scan coordinate descent methods (see e.g.\ Theorem $2$ in \cite{nesterov2012efficiency} with $\alpha=0$).

On the other hand, in the convex optimization case, the potential gain of coordinate descent relative to standard gradient descent is only given by the dependence on a better condition number, i.e.\ the replacement of $\kappa$ with $\kappa^* \leq \kappa$ in the convergence rate. Instead, in the sampling context, beyond the replacement of $\kappa$ with $\kappa^*$ in Theorem \ref{theo:contraction_KL_main}, there is an additional gain that directly depends on $d$.
Indeed, when applied to sampling, common SDE time-discretization techniques (such as the Euler-Maruyama one) introduce a bias in the invariant distribution of the induced Markov chain. 
This bias is usually either controlled by taking a small step-size (which decreases with $d$) or removed by adding a Metropolis-Hastings accept-reject step (which, however, also requires the step-size to decrease with $d$ to control the acceptance rate). In both cases, which are often referred to as \emph{unadjusted} or \emph{Metropolis-adjusted}, the resulting gradient-based MCMC algorithm requires a number of target evaluations for each $\epsilon$-accurate sample that grows with $d$ (see Table \ref{table:computational_cost}). By contrast, GS does not introduce any discretization error in its invariant distribution and, under Assumption \ref{ass:comp_cost}, it does not pay any dimensionality price (beyond the potential implicit dependence of $\kappa^*$ on $d$). 

\section{Proof of the main result}
\label{sec:proof_main}

We now turn into the proof of Theorem~\ref{theo:functional_main}. 
For a map $T : \R^d \to \R^d$, we write $T_m(x)$ and $T_{-m}(x)$ for respectively the component on $\sX_m$ and on $\sX_{-m}$ of $T(x)$. On the other hand, we use superscripts $(T^m)_{m=1,\dots,M}$ in order to index a family of maps by the parameter $m$. Recall that we use the same symbol to denote a measure and its density with respect to the Lebesgue measure. 

\subsection{Proof strategy}

We will find a collection of maps, $\hat{T}^m : \R^d \to \R^d$ for $m=1,\dots,M$, such that $\hat{T}^m_{-m}(x)=x_{-m}$ (that is, $\hat{T}^m$ only changes the $m$-th coordinate of $x$) and such that
\begin{equation}
\label{eq:to_prove}
\frac{1}{M} \sum_{m=1}^M \KL\left( \left. (\hat{T}^m)_\# \mu \right| \pi \right) \leq \left( 1 - \frac{1}{\kappa^* M} \right) \KL(\mu | \pi).   
\end{equation}
Since $(\hat{T}^m)_\# \mu$ has the same marginals as $\mu$ on $\sX_{-m}$, by the variational characterization of Lemma~\ref{lm:variational_characterization_Pm} we have $\KL((\hat{T}^m)_\# \mu|\pi) \geq \KL(\mu_{-m}|\pi_{-m})$ and it follows that~\eqref{eq:to_prove} implies the desired result.

For $\mu \in \sP(\R^d)$ with $\KL(\mu|\pi) < + \infty$, we can decompose $\KL(\mu | \pi) = \U(\mu) + \mathcal{H}(\mu)$ into its potential and entropy part defined as
\begin{equation*}
\U(\mu) = \int_{\R^d} U(x) \,   \mu(\d x), \qquad 
\mathcal{H}(\mu) =
\int_{\R^d} \log  \mu(x)  \,  \mu(\d x)\,.
\end{equation*}
We will prove a stronger version than \eqref{eq:to_prove}, namely we will find maps $(\hat{T}^m)_{m=1,\dots,M}$ such that
\begin{align*}
\frac{1}{M} \sum_{m=1}^M \U \left(  (\hat{T}^m)_\# \mu \right) & \leq \left( 1 - \frac{1}{\kappa^* M} \right)\U(\mu) + \frac{1}{\kappa^* M} \U(\pi), \\
\frac{1}{M} \sum_{m=1}^M \sH \left(  (\hat{T}^m)_\# \mu \right) & \leq \left( 1 - \frac{1}{\kappa^* M} \right)\mathcal{H}(\mu) + \frac{1}{\kappa^* M} \mathcal{H}(\pi).
\end{align*}
Summing the two inequalities and using $\U(\pi) + \mathcal{H}(\pi) = \KL(\pi|\pi) = 0$ yields~\eqref{eq:to_prove} and thus the conclusion.
While the potential part $\U$ is relatively easy to handle and is an adaptation of arguments in the literature on coordinate-wise optimization \cite{RT14,W15}, the entropy part $\sH$ is more delicate and will rely on a particular choice of maps $(\hat{T}^m)_{m=1,\dots,M}$.

\subsection{The potential part}

We start with some implications of Assumption~\ref{asmp:convex_smooth_block} about how the potential function $U$ changes when updating one coordinate at the time. Recall that $U$ decomposes as $U(x) = U_0(x) + \sum_{m=1}^M U_m(x_m)$, and that $U_0$ is block-smooth and $\lambda^*$-convex. Specifically we recall, see~\cite[Eq. (6)]{RT14}, that the block-smoothness of $U_0$ implies for any $x,y$
\begin{equation*}
U_0(y_m,x_{-m}) \leq U_0(x) + \nabla_{m} U_0(x)^\top (y_m - x_m) + \frac{L_m}{2} \| y_m - x_m \|^2.
\end{equation*}
On the other hand the $\lambda^*$-convexity under the metric $\| x \|_L^2 = \sum_{m=1}^M L_m \| x_m \|^2$ implies for any $x,y$, 
\begin{equation*}
U_0(y) \geq U_0(x) + \nabla U_0(x)^\top (y - x) + \frac{\lambda^*}{2} \| y - x \|_L^2,
\end{equation*}
see~\cite[Eq. (9)]{RT14}. It also implies that if $y = (1-t)x + t\bar{x}$ then 
\begin{equation*}
U_0(y) \leq (1-t) U_0(x) + t U_0(\bar{x}) - \frac{\lambda^* t(1-t)}{2} \| x - \bar{x} \|_L^2,
\end{equation*}
see~\cite[Eq. (11)]{RT14}. Armed with these three inequalities we can deduce the following bound, which is inspired by arguments previously used to analyze convergence of coordinate-wise optimization methods ~\cite[Theorem 4]{W15}; see also similar results in a previous work~\cite{LZ2024CAVI} by the last two authors.

\begin{lemma}
\label{lm:potential_U_key}
Under Assumption~\ref{asmp:convex_smooth_block}, let $x, \bar{x}\in\R^d$ and $y = (1-\lambda^*) x + \lambda^* \bar{x}$. Then
\begin{equation}
\label{eq:potential_U_key}
\frac{1}{M} \sum_{m=1}^M U(y_m, x_{-m}) \leq \left( 1 - \frac{1}{M\kappa^*} \right) U(x) + \frac{1}{M \kappa^*} U(\bar{x}).
\end{equation}
\end{lemma}

\begin{proof}
We write $U^{(s)}(x) = \sum_{m=1}^M U_m(x_m)$ for the separable part, so that $U = U_0 + U^{(s)}$. 
We first start with a computation valid for any $x,y$: using block-smoothness and then convexity of $U_0$ we have
\begin{align*}
\frac{1}{M} \sum_{m=1}^M U_0(y_m, x_{-m}) & \leq \frac{1}{M} \sum_{m=1}^M \left( U_0(x) + \nabla_{m} U_0(x)^\top (y_m - x_m) + \frac{L_m}{2} \|y_m - x_m \|^2 \right) \\
& = \frac{M-1}{M} U_0(x) + \frac{1}{M} \left( U_0(x) + \nabla U_0(x)^\top (y-x) + \frac{1}{2} \| y-x \|_L^2 \right) \\
& \leq \frac{M-1}{M} U_0(x) + \frac{1}{M} \left( U_0(y) + \frac{1 - \lambda^*}{2} \| y-x \|_L^2 \right). 
\end{align*}
On the other hand for the separable part we have the algebraic identity:
\begin{align}
\notag
\frac{1}{M} \sum_{m=1}^M U^{(s)}(y_m, x_{-m}) & = \frac{1}{M} \sum_{m=1}^M \left( U_m(y_m) + \sum_{m'\neq m} U_{m'}(x_{m'}) \right) \\
\label{eq:algebra_separable}
& = \frac{M-1}{M} \sum_{m=1}^M U_m(x_m) + \frac{1}{M} \sum_{m=1}^M U_m(y_m) = \frac{M-1}{M} U^{(s)}(x) + \frac{1}{M} U^{(s)}(y).
\end{align}
Summing with the part $U_0$, we get
\begin{equation}
\label{eq:aux_potential_without_cancellations}
\frac{1}{M} \sum_{m=1}^M U(y_m, x_{-m}) \leq \frac{M-1}{M} U(x) + \frac{1}{M} \left( U(y) + \frac{1 - \lambda^*}{2} \| y-x \|_L^2 \right).
\end{equation}
At this point, as $U^{(s)}$ is convex while $U_0$ is $\lambda^*$-convex, we have that $U = U_0 + U^{(s)}$ is $\lambda^*$-convex. Thus, using $y = (1-\lambda^*) x + \lambda^* \bar{x}$ so that $\| x-y \|_L = \lambda^* \| x - \bar{x} \|_L$, we have 
\begin{align*}
\frac{1}{M} & \sum_{m=1}^M U(y_m, x_{-m})   \leq \frac{M-1}{M} U(x) + \frac{1}{M} \left( U(y) + \frac{1-\lambda^*}{2} \| y-x \|_L^2 \right) \\
 & \leq \frac{M-1}{M} U(x) + \frac{1}{M} \left( (1-\lambda^*)U(x) + \lambda^* U(\bar{x}) - \frac{ \lambda^* (1-\lambda^*) \lambda^*}{2}\| x - \bar{x} \|_L^2 + \frac{{\lambda^*}^2 (1-\lambda^*)}{2}  \| x-\bar{x} \|_L^2 \right).
\end{align*}
The terms with the distance cancel, and we conclude with $\kappa^* = 1/\lambda^*$.
\end{proof}

Lemma \ref{lm:potential_U_key} is directly lifted to the integrated potential energy $\U$.

\begin{prop}
\label{prop:conclusion_estimate_U}
Under Assumption~\ref{asmp:convex_smooth_block}, let $\mu\in\sP(\R^d)$ and $T : \R^d \to \R^d$ be a measurable map. For $m=1, \ldots, M$ define $\hat{T}^m : \R^d \to \R^d$ by $\hat{T}^m_{-m}(x) = x_{-m}$ and $\hat{T}^m_{m}(x) = (1-\lambda^*) x_m + \lambda^* T_m(x)$. Then 
\begin{equation*}
\frac{1}{M} \sum_{m=1}^M \U \left(  (\hat{T}^m)_\# \mu \right) \leq \left( 1 - \frac{1}{\kappa^* M} \right)\U(\mu) + \frac{1}{\kappa^* M} \U(T_\# \mu).
\end{equation*}
\end{prop}
\begin{proof}
We evaluate~\eqref{eq:potential_U_key} when $\bar{x} = T(x)$. Note in this case that $(y_m,x_{-m}) = \hat{T}^m(x)$. Therefore, when integrating with respect to $\mu$, we have
\[
\frac{1}{M} \sum_{m=1}^M\int_{\R^d}U\left(\hat{T}^m(x) \right) \mu(\d x) \leq \left( 1 - \frac{1}{\kappa^* M} \right)\U(\mu) + \frac{1}{\kappa^* M} \int_{\R^d}U\left(T(x) \right) \mu(\d x). 
\]
This yields the result by definition of the pushforward.
\end{proof}

Proposition \ref{prop:conclusion_estimate_U} relates the potential energies of $((\hat{T}^m)_\# \mu)_{m=1,\dots,M}$ to the one of $T_\# \mu$.
Note that $T$ can be \emph{any} map, it does not have to be e.g.\ the optimal transport map. This flexibility will be important to handle the entropy part. On the other hand $\hat{T}^m_{m}(x)$ is a convex combination of $x_m$ and $T_m(x)$ with weights $(1-\lambda^*,\lambda^*)$ depending on the convexity parameters of $U$.

\subsection{The entropy part}
\label{sec:proof_entropy_part}

We now turn to the entropy part. Note that the computations below do not use the assumption on the potential $U$ as they deal with the entropy $\mathcal{H}$; they only rely on the structure of the transport maps.

If $T$ is a $C^1$-diffeomorphism then for any $\mu$ with $\mathcal{H}(\mu) < + \infty$ we have 
\begin{equation}
\label{eq:formula_entropy_pf}
\mathcal{H}(T_\# \mu) = \mathcal{H}(\mu) - \int_{\R^d} \log |\det DT(x)| \, \mu(\d x)\,,
\end{equation}
where $DT(x)$ denotes the Jacobian matrix of $T$ and $|\det DT(x)|$ is the absolute value of its determinant.
The above follows from the change of variables formula $(T_\# \mu)(x) = \mu(T^{-1}(x)) / |\det DT(T^{-1}(x))|$. To handle the entropy we will rely on a specific class of transport maps: not optimal transport ones but rather triangular ones.

\begin{definition}
A map $T : \R^d \to \R^d$ is said triangular if, for any $i = 1, \ldots, d$, the $i$-th component of $T(x)$ depends only on $(x_1, \ldots, x_i)$ but not on $(x_{i+1}, \ldots, x_d)$. We say in addition that it is increasing if $x_i \mapsto T_i(x_i,\bar{x}_{-i})$ is increasing for any $\bar{x}\in\R^d$ and any $i=1, \ldots d$.  
\end{definition}

 If $T$ is a triangular map which is also differentiable, then its Jacobian matrix is lower triangular. If in addition it is increasing, then the diagonal coefficients of the Jacobian matrix are non-negative.

There is a canonical triangular increasing transport map between continuous measures: it was introduced independently for statistical testing by Rosenblatt~\cite{rosenblatt1952remarks}, and to prove geometric inequalities by Knothe~\cite{knothe1957contributions}. It is usually called the Knothe-Rosenblatt map: we refer to e.g.~\cite[Section 2.3]{santambrogio2015optimal} for a modern presentation of the following theorem.  

\begin{theorem}
\label{theo:existence_KR}
Let $\mu, \nu \in \sP(\R^d)$ be absolutely continuous with respect to the Lebesgue measure.
Then there exists a triangular and increasing map $\bar{T} : \R^d \to \R^d$ such that $\bar{T}_\# \mu = \nu$. 
\end{theorem}

 In the sequel we will also assume that the maps are $C^1$-diffeomorphisms in order to use~\eqref{eq:formula_entropy_pf}. Though the Knothe-Rosenblatt map given by Theorem~\ref{theo:existence_KR} is not always smooth, we will later reason by approximation if it is not the case: see Proposition~\ref{prop:KR_smooth} and the arguments which follow below.

\begin{lemma}
\label{lm:entropy_triangular_eq}
Let $\mu\in\sP(\R^d)$ and $T : \R^d \to \R^d$ be a triangular $C^1$-diffeomorphism. For $m=1, \ldots, M$ define $\hat{T}^m : \R^d \to \R^d$ by $\hat{T}_{-m}^m(x) = x_{-m}$ and $\hat{T}_{m}^m(x) = T_m(x)$. Then
\begin{equation*}
\frac{1}{M} \sum_{m=1}^M \sH\left( (\hat{T}^m)_\# \mu \right) = \frac{M-1}{M} \mathcal{H}(\mu) + \frac{1}{M} \mathcal{H}(T_\# \mu). 
\end{equation*}
\end{lemma}

\begin{proof}
From~\eqref{eq:formula_entropy_pf} the left hand side reads
\begin{align*}
\frac{1}{M} \sum_{m=1}^M \sH\left( (\hat{T}^m)_\# \mu \right) &= \frac{1}{M}\sum_{m = 1}^M\left( \mathcal{H}(\mu) - \int_{\R^d} \log |\det D\hat{T}^m(x)| \, \mu(\d x)\right)\\
&= \frac{M-1}{M} \mathcal{H}(\mu) + \frac{1}{M}\left(\mathcal{H}(\mu)- \sum_{m = 1}^M\int_{\R^d} \log |\det D\hat{T}^m(x)| \, \mu(\d x)\right).
\end{align*}
Given the definition of $\hat{T}^m$, we have that $\det D\hat{T}^m(x)$ coincides with $\det [ DT(x) ]_{mm}$, with $[DT(x)]_{mm}$ the $d_m \times d_m$ diagonal block of the matrix $DT(x)$. Thus, again with~\eqref{eq:formula_entropy_pf} we have 
\begin{multline}
\label{eq:proof_entropy_full}
\frac{1}{M} \sum_{m=1}^M \sH\left( (\hat{T}^m)_\# \mu \right) = \frac{M-1}{M} \mathcal{H}(\mu) + \frac{1}{M} \mathcal{H}(T_\# \mu) \\
+ \left\{ \frac{1}{M} \int_{\R^d} \left( \log |\det DT(x)| - \sum_{m=1}^M \log | \det [ DT(x) ]_{mm}|  \right) \, \mu(\d x) \right\}.
\end{multline}
At this point we use the triangular structure: for a single $x \in \R^d$, as $D T(x)$ is triangular, we have
\begin{equation*}
|\det DT(x)| =  \prod_{m=1}^M |\det [D T]_{mm}(x)|. 
\end{equation*}
Taking the logarithm and integrating with respect to $x$, we see that the term in brackets in~\eqref{eq:proof_entropy_full} vanishes and we have the result.
\end{proof}

We expect Lemma \ref{lm:entropy_triangular_eq} to extend beyond the case of $C^1$-diffeomorphisms. However, the $C^1$ assumption makes the proof more direct and is useful for Lemma~\ref{lm:convexity_along_KR} below. 

\begin{remark}
\label{rmk:why_not_OT}
Note the algebraic similarity of the result of Lemma~\ref{lm:entropy_triangular_eq} with~\eqref{eq:algebra_separable}: roughly speaking, when we push measures through triangular maps, the entropy behaves like a separable function. Here the triangular structure is important: if we take $T$ an optimal transport map, then Equation~\eqref{eq:proof_entropy_full} is still valid, but in this case $DT$ is a symmetric matrix and we cannot find a sign to the term in brackets. The only a priori information is $DT \leq M \, \mathrm{diag}([DT]_{11}, \ldots, [DT]_{MM})$, see~\cite[Lemma 1]{nesterov2012efficiency}, thus the term in brackets is bounded by $\log (M)$. But we cannot say more.  
\end{remark}

To move forward, and given the shape of the maps $\hat{T}^m$ in Proposition~\ref{prop:conclusion_estimate_U}, we need to evaluate $\mathcal{H}(T_\# \mu)$ when $T$ is a linear combination of the identity and a triangular map. 

\begin{lemma}
\label{lm:convexity_along_KR}
Let $\mu\in\sP(\R^d)$ and $T : \R^d \to \R^d$ be triangular, increasing, and a $C^1$-diffeomorphism. Then the function $t \mapsto \sH \left([(1-t)\Id+ t T]_\# \mu \right)$ is convex for $t \in [0,1]$.
\end{lemma}

 If $T$ were an optimal transport map instead of a triangular increasing map, then this lemma corresponds to the celebrated displacement convexity of the entropy~\cite{mccann1997convexity}, see also \cite[Chapters 16\&17]{villani2009optimal} and \cite[Chapter 7]{santambrogio2015optimal}. The proof of this lemma is actually the same as in optimal transport, as it is enough that $D T$ has non-negative eigenvalues for the result to hold.  

\begin{proof}
Let $\sigma_1(x), \ldots, \sigma_d(x)$ be the eigenvalues of $D T(x)$ for $x \in \R^d$. As $D T(x)$ is a triangular matrix these correspond to diagonal coefficients, that is, $\sigma_i(x) = \partial_{x_i} T_i(x)$, and they are strictly positive since the map $x_i \mapsto T_i(x_i,\bar{x}_{-i})$ is assumed to be increasing for any $\bar{x}$ and $i$. Thus, starting from~\eqref{eq:formula_entropy_pf},
\begin{align*}
\sH \left([(1-t)\Id+ t T]_\# \mu \right) & = \mathcal{H}(\mu) - \int_{\R^d} \log | \det [(1-t)\Id+ tD T(x)] | \, \mu(\d x) \\
& = \mathcal{H}(\mu) - \sum_{i=1}^d \int_{\R^d} \log( (1-t) + t \sigma_i(x)) \, \mu(\d x). 
\end{align*}
(Note that we removed the absolute value thanks to the non-negativity of $\sigma_i(x)$.) As $t \mapsto - \log( (1-t) + t \sigma_i(x))$ is convex, we see that the function $t \mapsto \sH \left([(1-t)\Id+ t T]_\# \mu \right)$ is also convex. 
\end{proof}

 We now are in position to conclude our study of the entropy part. 

\begin{prop}
\label{prop:conclusion_estimate_H}
Let $\mu\in\sP(\R^d)$, $t\in[0,1]$ and $T : \R^d \to \R^d$ be triangular, increasing and a $C^1$-diffeomorphism. For $m=1, \ldots, M$ define $\hat{T}^m : \R^d \to \R^d$ by $\hat{T}^m_{-m}(x) = x_{-m}$ while $\hat{T}^m_{m}(x)= (1-t) x_m + t T_m(x)$. Then
\begin{equation*}
\frac{1}{M} \sum_{m=1}^M \sH \left(  (\hat{T}^m)_\# \mu \right) \leq \left( 1 - \frac{t}{M} \right)\mathcal{H}(\mu) + \frac{t}{ M} \mathcal{H}(T_\# \mu).
\end{equation*}
\end{prop}

\begin{proof}
We use first Lemma~\ref{lm:entropy_triangular_eq} and then the convexity given by Lemma~\ref{lm:convexity_along_KR}:
\begin{align*}
\frac{1}{M} \sum_{m=1}^M \sH \left(  (\hat{T}^m)_\# \mu \right) & =   \frac{M-1}{M} \mathcal{H}(\mu) + \frac{1}{M} \sH \left([(1-t)\Id+ t T]_\# \mu \right) \\
& \leq  \frac{M-1}{M} \mathcal{H}(\mu) + \frac{1}{M} \biggl( (1-t) \mathcal{H}(\mu) + t \mathcal{H}(T_\# \mu) \biggr)\\
& = \left( 1 - \frac{t}{M} \right)\mathcal{H}(\mu) + \frac{t}{ M} \mathcal{H}(T_\# \mu). \qedhere
\end{align*}
\end{proof}

Compared to Proposition~\ref{prop:conclusion_estimate_U}, the result for $\sH$ holds for a specific class of maps (the triangular and increasing ones), but then $T_m(x)$ is any convex combination of $x_m$ and $\bar{T}_m(x)$. 

\subsection{Conclusion of the proof of the main result}

\begin{proof}[Proof of Theorem~\ref{theo:functional_main} if the Knothe-Rosenblatt map is a diffeomorphism]
We first prove the result if the Knothe-Rosenblatt map $\bar{T}$ between $\mu$ and $\pi$, given by Theorem~\ref{theo:existence_KR}, is a $C^1$-diffeomorphsim. In this case for $m=1, \ldots, M$ define $\hat{T}^m : \R^d \to \R^d$ by $\hat{T}^m_{-m}(x) = x_{-m}$ and $\hat{T}^m_{m}(x) = (1-\lambda^*) x_m + \lambda^* \bar{T}_m(x)$. With Proposition~\ref{prop:conclusion_estimate_U}, as $\bar{T}_\# \mu = \pi$ we find 
\begin{equation*}
\frac{1}{M} \sum_{m=1}^M \U \left(  (\hat{T}^m)_\# \mu \right) \leq \left( 1 - \frac{1}{\kappa^* M} \right)\U(\mu) + \frac{1}{\kappa^* M} \U(\pi).
\end{equation*}
On the other hand, for the entropy part we apply Proposition~\ref{prop:conclusion_estimate_H} with $t = \lambda^* = 1/\kappa^*$ to obtain
\begin{equation*}
\frac{1}{M} \sum_{m=1}^M \sH \left(  (\hat{T}^m)_\# \mu \right) \leq \left( 1 - \frac{1}{\kappa^* M} \right)\mathcal{H}(\mu) + \frac{1}{\kappa^* M} \mathcal{H}(\pi).
\end{equation*}
Summing these two inequalities yields~\eqref{eq:to_prove}, hence the conclusion.
\end{proof}

To move to the general case, we will need the following result on assumptions guaranteeing the regularity of the Knothe-Rosenblatt map.  

\begin{prop}
\label{prop:KR_smooth}
Assume $\mu,\nu$ are supported on a hypercube $Q_R=[-R,R]^d$, with $0<R<\infty$, and their density is of class $C^1$ on $Q_R$, and uniformly strictly positive. 
Then the Knothe-Rosenblatt map $\bar{T}$ from Theorem~\ref{theo:existence_KR} is a $C^1$-diffeomorphism. 
\end{prop}

\begin{proof}
Indeed with these assumptions for any $i$ the conditional densities $\mu(x_{i+1}, \ldots, x_d | x_{1}, \ldots x_{i} )$ and $\nu(x_{i+1}, \ldots, x_d | x_{1}, \ldots x_{i} )$ have a density which is strictly positive and of class $C^1$ on the hypercube. The smoothness of $\bar{T}$ follows by~\cite[Remark 2.19]{santambrogio2015optimal}.
\end{proof}

Given Proposition~\ref{prop:KR_smooth}, we can proceed by approximation, that is, we approximate $\mu$ and $\pi$ by $C^1$ measures supported on $Q_R$ and then exploit the regularity of the associated Knothe-Rosenblatt map.

\begin{proof}[Proof of Theorem~\ref{theo:functional_main} in the general case]

Without loss of generality, fix $\mu$ with $\KL(\mu| \pi) < + \infty$. 
It is easy to check that, if we approximate $U$ by a sequence $(U^p)_{p=1,2,\dots}$ obtained by convolution with smooth and compactly supported kernel, then $U^p$ satisfies Assumption~\ref{asmp:convex_smooth_block} with the same $\kappa^*$ for any $p$. Moreover, it can be done in such a way that $U^p$ is of class $C^1$ and converges uniformly on compact sets to $U$ as $p\to\infty$. 
Denote by $\mu^R$, $\pi^R$ and $\pi^{R,p}$  the restriction of $\mu$, $\pi$ and $\exp(-U^p)$ on $Q_R$, normalized to be probability distributions; and by $U^{R}$ and $U^{R,p}$ the negative log densities of $\pi^R$ and $\pi^{R,p}$ on $Q_R$.

Similarly, for each $R>0$, let $(\mu^{R,p})_{p=1,2,\dots}$ be a sequence of probability distributions, whose densities are of class $C^1$ and strictly positive, and which converges weakly to $\mu^R$ as $p \to + \infty$ and such that $\limsup_p \mathcal{H}(\mu^{R,p}) \leq \mathcal{H}(\mu^R)$. This can be obtained by convolving $\mu^R$ with Gaussian with vanishing variance and then normalizing: as the variance of the Gaussian vanishes, the convolution asymptotically does not change the entropy~\cite[Lemma 3]{rioul2010information}, and neither does normalization as the mass outside of $Q_R$ converges to $0$.

By Proposition~\ref{prop:KR_smooth}, the Knothe-Rosenblatt map between $\mu^{R,p}$ and $\pi^{R,p}$ is a $C^1$-diffeomorphism. Thus, we can apply Theorem~\ref{theo:functional_main} to obtain
\begin{equation}\label{eq:before_limit}
\frac{1}{M} \sum_{m=1}^M \KL\left( \left. \mu^{R,p}_{-m} \right| \pi^{R,p}_{-m}\right) \leq \left( 1 - \frac{1}{\kappa^* M} \right) \KL\left(\left. \mu^{R,p} \right| \pi^{R,p}\right). 
\end{equation}
We then pass to the limit $p  \to + \infty$ followed by $R \to + \infty$ in \eqref{eq:before_limit}. 
For the right hand side, by the uniform convergence of the $U^{R,p}$ to $U^{R}$ and the weak convergence of $\mu^{R,p}$ to $\mu^{R}$ on $Q_R$, and on the other hand the estimate $\limsup_p \mathcal{H}(\mu^{R,p}) \leq \mathcal{H}(\mu^R)$, we have 
 \begin{align*}
\limsup_{p \to + \infty} \KL\left(\left. \mu^{R,p} \right| \pi^{R,p}\right) & = \limsup_{p \to + \infty} \left( \int_{Q_R} U^{R,p} \, \d \mu^{R,p} + \mathcal{H}(\mu^{R,p})  \right) \\
& \leq  \int_{Q_R} U^{R} \, \d \mu^{R} + \mathcal{H}(\mu^R)  = 
\KL\left(\left. \mu^R \right| \pi^R\right)\,. 
\end{align*}
By the monotone convergence theorem $\KL(\mu^R | \pi^R) \to \KL(\mu|\pi)$ as $R\to\infty$, so that the above implies
\begin{equation}
\label{eq:aux_approx}
\limsup_{R \to + \infty} \limsup_{p \to + \infty} \KL(\mu^{R,p}| \pi^{R,p}) \leq \KL(\mu|\pi).
\end{equation}
For the left hand side of \eqref{eq:before_limit}, from the weak convergence of $\mu^{R,p}$ and $\pi^{R,p}$ to respectively $\mu$ and $\pi$, we know that for each $m$ the marginals $\mu^{R,p}_{-m}$ and $\pi^{R,p}_{-m}$ converge weakly to respectively $\mu_{-m}$ and $\pi_{-m}$ as $p \to + \infty$ followed by $R \to + \infty$. As $\KL$ is jointly lower semi-continuous (see e.g.~\cite[Lemma 9.4.3]{AGS}), we deduce
\begin{equation}\label{eq:liminf_aux}
\KL(\mu_{-m}|\pi_{-m}) \leq \liminf_{R \to + \infty} \liminf_{p \to + \infty} \; \KL\left( \left. \mu^{R,p}_{-m} \right| \pi^{R,p}_{-m}\right). 
\end{equation}
Using~\eqref{eq:aux_approx} and~\eqref{eq:liminf_aux} to pass to the limit in~\eqref{eq:before_limit}, we obtain \eqref{eq:functional_main} as desired.
\end{proof}

\section{Rate of convergence in the non-strongly convex case}
\label{sec:non_strongly_convex}

Theorem~\ref{theo:contraction_KL_main} degenerates and becomes uninformative when $\lambda^* = 0$. 
This corresponds to the case where the potential is assumed to be smooth and convex, but not strongly convex. 
In this case, similarly to the case of optimization, we show that the convergence of $\mu^{(n)}$ to $\pi$ in $\KL$ degrades from exponential to being $\sO(1/n)$.
As above, we denote $\mu^{(n+1)} = \mu^{(n)} P^\mathrm{GS}$ for $n\geq 0$.

\begin{theorem}
\label{theo:non_strongly_convex}
Let $\pi$ satisfy Assumption~\ref{asmp:convex_smooth_block} but with $\lambda^* = 0$, and let $x^*$ be a mode of $\pi$. For $\mu^{(0)} \in \sP(\R^d)$, define $R$ by
\begin{equation*}
R^2 = \max \left( \KL(\mu^{(0)}|\pi) , \; 4 \, \sup_{n \geq 0} \,  \int_{\R^d} \| x -x^* \|^2_L \,  \mu^{(n)}(\d x)\right)\,.
\end{equation*}
Then for any $n \geq 0$ it holds that
\begin{equation*}
\KL(\mu^{(n)} | \pi) \leq \frac{2M R^2}{n+2M}\,.
\end{equation*}
\end{theorem}

\begin{proof}
We do a one-step bound on the decay of $\KL$. We recall the decomposition $\KL(\cdot | \pi) = \U + \sH$ of Section~\ref{sec:proof_main} and we first fix $\mu$ with $\KL(\mu|\pi) < +  \infty$. 

We first consider the potential part. Let $x, \bar{x}\in\R^d$, $t \in [0,1]$ and $y = (1-t) x + t \bar{x}$. Then \eqref{eq:aux_potential_without_cancellations} and the convexity of $U$ yield
\begin{equation}\label{eq:before_integration}
\begin{aligned}
\frac{1}{M} \sum_{m=1}^M U(y_m,x_{-m}) &\leq \frac{M-1}{M} U(x) + \frac{1}{M} U(y) + \frac{\| x - y \|_L^2}{2M}\\
& \leq  \left( 1 - \frac{t}{M} \right) U(x) + \frac{t}{M} U(\bar{x}) + \frac{t^2}{2M} \| x - \bar{x} \|_L^2.   
\end{aligned}
\end{equation}
Suppose $T : \R^d \to \R^d$ measurable such that $T_\# \mu = \pi$, and define for $m=1, \ldots, M$ the maps $\hat{T}^m : \R^d \to \R^d$ by $\hat{T}^m_{-m}(x) = x_{-m}$ and $\hat{T}^m_{m}(x) = (1-t) x_m + t T_m(x)$.
Then, if we take $\bar{x} = T(x)$ and integrate \eqref{eq:before_integration} with respect to $\mu(\d x)$, using the fact that $\hat{T}^m(x)=(y_m,x_{-m})$, we have  
\begin{equation*}
\frac{1}{M} \sum_{m=1}^M \U \left(  (\hat{T}^m)_\# \mu \right) \leq  \left( 1 - \frac{t}{M} \right) \U(\mu) + \frac{t}{M} \U(\pi) + \frac{t^2}{2M} \int_{\R^d} \| T(x) - x \|_L^2 \, \mu(\d x).
\end{equation*}
If $T$ is the Knothe-Rosenblatt map between $\mu$ and $\pi$ and it is a $C^1$-diffeomorphism, we can apply Proposition~\ref{prop:conclusion_estimate_H} and, summing with the above, obtain for any $t \in [0,1]$, 
\begin{equation}
\label{eq:aux_proof_non_convex}
\frac{1}{M} \sum_{m=1}^M \KL \left( \left.  (\hat{T}^m)_\# \mu \right| \pi \right) \leq  \left( 1 - \frac{t}{M} \right) \KL(\mu|\pi) + \frac{t^2}{2M} \int_{\R^d} \| T(x) - x \|_L^2 \, \mu(\d x).
\end{equation}
We now take $\mu = \mu^{(n)}$. The left hand side is connected to the entropy of $\mu^{(n)} P^\mathrm{GS} = \mu^{(n+1)}$ through Lemma~\ref{lm:variational_characterization_Pm} as expanded in the proof of Theorem~\ref{theo:functional_main}. 
For the right hand side we use $\|T(x) - x \|_L^2 \leq 2\| T(x) - x^* \|_L^2 + 2\| x -x^* \|_L^2$, and that $T_\# \mu^{(n)} = \pi$. By definition $2 \int \| x -x^* \|_L^2 \,  \mu^{(n)}(\d x) \leq R^2/2$ while passing to the limit $n \to + \infty$ we also have $2 \int \| x -x^* \|_L^2 \,  \pi(\d x) \leq R^2/2$. Thus~\eqref{eq:aux_proof_non_convex} yields
\begin{align*}
\KL \left( \left. \mu^{(n+1)} \right| \pi \right) & \leq \frac{1}{M} \sum_{m=1}^M \KL\left( \left. \mu^{(n)}_{-m} \right| \pi_{-m} \right) \\
& \leq  \frac{1}{M} \sum_{m=1}^M \KL \left( \left.  (\hat{T}^m)_\# \mu^{(n)} \right| \pi \right) \leq \left( 1 - \frac{t}{M} \right) \KL(\mu^{(n)}|\pi) + \frac{t^2 R^2}{2M}.
\end{align*}
We choose $t = \KL(\mu^{(n)}|\pi) / R^2$ which belongs to $[0,1]$ as $\KL(\mu^{(n)}|\pi) \leq \KL(\mu^{(0)} | \pi) \leq R^2$. Thus we find 
\begin{equation*}
\KL \left( \left. \mu^{(n+1)} \right| \pi \right) \leq  \KL(\mu^{(n)}|\pi) \left( 1 - \frac{\KL(\mu^{(n)}|\pi)}{2MR^2}  \right).   
\end{equation*}
Note that we need the Knothe-Rosenblatt map between $\mu^{(n)}$ and $\pi$ to be a $C^1$-diffeomorphism to use Proposition~\ref{prop:conclusion_estimate_H}, if not we resort to an approximation argument as in the proof of Theorem~\ref{theo:functional_main}.

This inequality easily implies the rate announced in the theorem. In particular, $u_n = \KL(\mu^{(n)}|\pi)$ is a non-increasing sequence which satisfies $u_{n+1} \leq u_n (1 -u_n/(2MR^2))$. This implies
\begin{equation*}
\frac{1}{u_{n+1}} - \frac{1}{u_n} = \frac{u_n - u_{n+1}}{u_{n+1} u_n} \geq \frac{1}{2MR^2} \frac{u_n}{u_{n+1}} \geq \frac{1}{2MR^2}.
\end{equation*}
Summing from over $n$ yields $1/u_{n} - 1/u_0 \geq n/(2MR^2)$ and we conclude from $u_0 \leq R^2$.
\end{proof}

The bound in Theorem \ref{theo:non_strongly_convex} resembles the ones available for coordinate descent in Euclidean spaces, see e.g.~\cite[Eq.\ (2.14)]{nesterov2012efficiency}. In the latter case, however, the term $R^2$ can be bounded with the diameter of the sub-level sets of the target function. By contrast in our case, if $\pi$ is only log-concave but not strongly log-concave, a bound on $\KL(\mu|\pi)$ does not control the second moment of $\mu$ so that it is unclear if $\KL(\mu^{(0)}|\pi) < + \infty$ implies $R < + \infty$. However if we make the stronger assumption of a warm start (see Corollary~\ref{crl:GS_mix_time} and the discussion following it for more comments), we can guarantee that $R < + \infty$ as follows.

\begin{corollary}
\label{crl:non_strongly}
Let $\pi$ satisfy Assumption~\ref{asmp:convex_smooth_block} but with $\lambda^* = 0$, and $x^*$ be a mode of $\pi$. Assume that $\mu^{(0)}\in\sP(\R^d)$ satisfies $\mu^{(0)}(A) \leq C \pi(A)$ for some $C<\infty$ and any $A\in\mathcal{B}(\R^d)$. Then for any $n \geq 0$, 
\begin{equation*}
\KL(\mu^{(n)} | \pi) \leq \frac{2M}{n+2M} \, \max \left( \log(C), \, 4 C \int_{\R^d} \| x - x^* \|_L^2 \, \pi(\d x)  \right).
\end{equation*}
\end{corollary}

\begin{proof}
Let $f^{(n)}$ be the density of $\mu^{(n)}$ with respect to $\pi$. By assumption $f^{(0)} \leq C$, and as $f^{(n+1)} = P^\mathrm{GS} f^{(n)}$, we see easily that $f^{(n)} \leq C$ for any $n$. Thus for any $n$ 
\begin{equation*}
\int_{\R^d} \| x - x^* \|^2_L \,  \mu^{(n)}(\d x) \leq C \int_{\R^d} \| x - x^* \|^2_L \,  \pi(\d x),
\end{equation*}
while on the other hand $\KL(\mu^{(0)} | \pi) \leq \log(C)$. The result follows from Theorem~\ref{theo:non_strongly_convex}.
\end{proof}

Thus in the limit of large warm start constant $C$, we can guarantee $\KL(\mu^{(n)}|\pi) \leq \epsilon$ by taking $n= \mathcal{O}(MC/\epsilon)$, where $\mathcal{O}$ hides a multiplicative constant depending only on the second moment of $\pi$ in the metric $\| \cdot \|_L$. We do not claim that the linear dependence with respect to $C$ is tight, and we leave the development of better bounds on $R$ to future work.

\section{Metropolis-within-Gibbs and general coordinate-wise kernels}
\label{sec:MwG}

In many contexts, it is common to replace the exact full conditional updates in Algorithm \ref{alg:rs-gs} with more general conditional updates. This leads to general ``coordinate-wise'' Markov transition kernels of the form
\begin{align}\label{eq:mwg_def}
P^\mathrm{MwG}&=\frac{1}{M}\sum_{m=1}^M \bar{P}_m        
&
\bar{P}_m(x,A) = \int_{\sX_m} \1_A(y_m,x_{-m}) \bar{P}_m^{x_{-m}}(x_m,\d y_m)
\end{align}
where $\bar{P}_m^{x_{-m}}$ is an arbitrary $\pi(\cdot|x_{-m})$-invariant Markov kernel on $\sX_m \times \mathcal{B}(\sX_m)$. Algorithmically, a sample $Y\sim P^\mathrm{MwG}(x,\cdot)$ is obtained as:
\begin{enumerate}
\item pick a coordinate $m\in\{1,\dots,M\}$ at random,
\item sample $Y_m\sim \bar{P}_m^{x_{-m}}(x_m,\cdot)$ and set $Y_{-m}=x_{-m}$.
\end{enumerate}
A common choice for $\bar{P}_m^{x_{-m}}$ is given by Metropolis-Hastings (MH) kernels: they combine a user-defined proposal kernel, $Q_m^{x_{-m}}(x_m,\cdot) \in \sP(\sX_{m})$ with an accept/reject step to have $\pi(\cdot|x_{-m})$ as invariant distribution. Specifically they are defined as
\begin{equation}\label{eq:operator_MH}
\begin{aligned}
\bar{P}_m^{x_{-m}}(x_m, A) 
=
\int_A\alpha_m^{x_{-m}}(x_m, y_m)&Q_m^{x_{-m}}(x_m, \d  y_m)\\
&+\delta_{x_m}(A)\int_{\sX_m} \left[1-\alpha_m^{x_{-m}}(x_m, y_m)\right]Q_m^{x_{-m}}(x_m, \d  y_m),
\end{aligned}
\end{equation}
for every $x\in\R^d$ and $A \in \mathcal{B}(\sX_m)$, with
\begin{equation*}
\alpha^{x_{-m}}_m(x_m, y_m) 
=
 \min \left\{ 1, 
\frac{ 
 \pi(\d y_m | x_{-m}) Q_m^{x_{-m}}(y_m, \d x_m)
 }{
 \pi(\d x_m | x_{-m}) Q_m^{x_{-m}}(x_m, \d y_m)
 }
 \right\}\,.
\end{equation*}
The resulting algorithms are usually called Metropolis-within-Gibbs (MwG) schemes, which is why we use the MwG acronym, even if in Theorem~\ref{thm:MwG} below we do not require $\bar{P}_m^{x_{-m}}$ to be of MH type.
Note that $P^\mathrm{MwG}$ coincides with $P^\mathrm{GS}$ if $\bar{P}_m^{x_{-m}}(x_m,\cdot)=\pi(\cdot|x_{-m})$.
MwG schemes are often used when they are more convenient to implement or computationally faster relative to the exact Gibbs updates discussed in Section \ref{sec:implementation}.

We have the following result for $P^\mathrm{MwG}$.
The proof of part (b) relies on 
a minor variation of \cite[Corollary 11]{qin2023spectral}, whose direct proof we report for completeness.
\begin{theorem}\label{thm:MwG}
Let $\pi$ satisfy Assumption~\ref{asmp:convex_smooth_block}.
\begin{enumerate}
\item[(a)] Assume $\bar{P}_m^{x_{-m}}$ are $\pi(\cdot|x_{-m})$-invariant kernels such that: for some $\beta > 0$, there holds
\begin{equation}\label{eq:KL_contr_cond_updates}
\KL(\nu \bar{P}_m^{x_{-m}}| \pi(\cdot|x_{-m})) \leq (1-\beta) \KL(\nu | \pi(\cdot|x_{-m}))
\end{equation}
for all $m = 1,\dots,M$, $x_{-m}\in \sX_{-m}$, and $\nu\in\mathcal{P}(\sX_m)$. Then for any $\mu\in\sP(\R^d)$
\begin{equation*}
\KL(\mu P^\mathrm{MwG} | \pi) \leq \left( 1 - \frac{\beta}{\kappa^* M} \right) \KL(\mu| \pi).
\end{equation*}
\item[(b)] Assume $\bar{P}_m^{x_{-m}}$ are $\pi(\cdot|x_{-m})$-reversible kernels such that: for some $\beta > 0$, there holds $\gap(\bar{P}_m^{x_{-m}})\geq \beta$ for all $m=1,\dots,M$ and $x_{-m} \in \sX_{-m}$. Then
$$
\gap(P^\mathrm{MwG})\geq \frac{\beta}{\kappa^* M}\,.
$$
\end{enumerate}
\end{theorem}
\begin{proof}
Consider first part (a), and let $\mu\in\mathcal{P}(\R^d)$.
By the convexity of KL and the chain rule in \eqref{eq:KL_chain_rule}, combined with the fact that $(\mu \bar{P}_m)_{-m}=\mu_{-m}$ (since $\bar{P}_m$ moves only the $m$-th coordinate), we have
\begin{align*}
\KL(\mu| \pi)-&\KL(\mu P^\mathrm{MwG}| \pi) 
\geq\frac{1}{M}\sum_{m=1}^M\KL(\mu| \pi)-\KL(\mu \bar{P}_m| \pi) 
\\
&
=\frac{1}{M}\sum_{m=1}^M
\E_{\mu_{-m}} \biggl[\KL(\mu(\cdot|X_{-m}) | \pi(\cdot|X_{-m}))
-
\KL(\mu(\cdot|X_{-m})\bar{P}_m^{X_{-m}}| \pi(\cdot|X_{-m}))\biggr]\,.
\end{align*}
By \eqref{eq:KL_contr_cond_updates} applied to $\nu=\mu(\cdot|x_{-m})$, and using then the formulation~\eqref{eq:functional_main_alternative} of our main result, 
\begin{equation*}
\KL(\mu| \pi)-\KL(\mu P^\mathrm{MwG}| \pi) \geq \frac{\beta}{M}\sum_{m=1}^M \E_{\mu_{-m}}[\KL(\mu(\cdot|X_{-m}) | \pi(\cdot|X_{-m}))] \geq   \frac{\beta}{\kappa^* M}\KL(\mu| \pi) \,,
\end{equation*}
which is our conclusion.

We now turn to part (b). 
The reversibility of $\bar{P}_m^{x_{-m}}$ implies the one of $P^\mathrm{MwG}$, so that we can apply the spectral gap definition in \eqref{eq:gap}. 
For $f \in L^2(\pi)$, by definition we have
\begin{align}
\langle f, (\Id-P^\mathrm{MwG})f \rangle_{\pi} 
&=
\frac{1}{M}\sum_{m = 1}^M
\langle f, (\Id-\bar{P}_m)f \rangle_{\pi} 
\nonumber\\&=
\frac{1}{M}\sum_{m = 1}^M\int_{\sX_{-m}} \langle f(\cdot, x_{-m}), (\Id-\bar{P}^{x_{-m}}_m)f(\cdot, x_{-m}) \rangle_{\pi(\cdot |  x_{-m})}\pi_{-m}(\d x_{-m})\,.\label{equality_proof_gap1}
\end{align}
By definition, $\gap(\bar{P}_m^{x_{-m}})\geq \beta$ implies 
\begin{align*}
\langle f(\cdot, x_{-m}), (\Id-\bar{P}^{x_{-m}}_m)f(\cdot, x_{-m}) \rangle_{\pi(\cdot| x_{-m})}\geq \beta \mathrm{Var}_{\pi(\cdot |  x_{-m})}(f(\cdot, x_{-m}))    
\end{align*}
where $\mathrm{Var}_{\pi(\cdot |  x_{-m})}(f(\cdot, x_{-m}))$ denotes $\mathrm{Var}\left(f(X) |  X_{-m} = x_{-m} \right)$ under $X\sim \pi$. 
Injecting the inequality above in \eqref{equality_proof_gap1} and using the expression~\eqref{eq:expression_Dir_PGS} of the Dirichlet form for $P^\mathrm{GS}$,
\begin{align*}
\langle f, (\Id-P^\mathrm{MwG})f \rangle_{\pi} 
&\geq 
\frac{\beta}{M} \sum_{m=1}^M  \int_{\sX_{-m}} \mathrm{Var}_{\pi(\cdot |  x_{-m})}(f(\cdot, x_{-m})) \pi_{-m}(\d x_{-m})
=
\beta \langle f, (\Id-P^\mathrm{GS})f \rangle_{\pi}\,,
\end{align*}
which implies $\gap(P^\mathrm{MwG})\geq \beta \gap(P^\mathrm{GS})$.
The result then follows by Corollary \ref{crl:variance_inequality}.
\end{proof}
Theorem \ref{thm:MwG} implies that, if the exact conditional updates in GS are replaced with Markov updates $\bar{P}_m^{x_{-m}}$ that contract at rate $(1-\beta)$ in $\KL$, or whose spectral gap is lower bounded by $\beta$, the total number of iterations required to converge is increased by at most a multiplicative factor $1/\beta$ when going from GS to MwG. 
Importantly, the dimensionality $d_m$ of $\pi(\cdot|x_{-m})$ is independent of $d$ (e.g.\ one can even take $d_m=1$), and thus it is reasonable to expect that in many contexts that $\beta$ does not depend directly on $d$.

We now give two examples of MH kernels to which Theorem \ref{thm:MwG} can be applied.
For the entropy contraction example (Corollary \ref{cor:IMH}), we rely on the following lemma, which provides a general (though very strong) sufficient condition for \eqref{eq:KL_contr_cond_updates} based on a simple minorization argument.
\begin{lemma}\label{lemma:minorizing}
Assume $
\bar{P}_m^{x_{-m}}(x_m, A) \geq \beta \pi(A | x_{-m})
$
for some $\beta > 0$  and any $x_m\in\sX_m$ and $A \in\mathcal{B}( \sX_m)$. Then \eqref{eq:KL_contr_cond_updates} holds with the same $\beta$.
\end{lemma}
\begin{proof}
We can rewrite $\bar{P}_m^{x_{-m}}(x_m,A)=\beta \pi(A| x_{-m})+(1-\beta) \hat{P}_m^{x_{-m}}(x_m,A)$ where $\hat{P}_m^{x_{-m}}$ is a $\pi(\cdot| x_{-m})$-invariant Markov kernel.
Then, by convexity of the relative entropy and $\pi(\cdot|x_{-m})$-invariance of $\hat{P}_m^{x_{-m}}$ (see also \cite[Lemma 9.4.5]{AGS}), for any $\nu\in\sP(\sX_m)$ we obtain
\begin{align*}
\qquad\quad
\KL(\nu \bar{P}_m^{x_{-m}}| \pi(\cdot|x_{-m})) &\leq 
(1-\beta)\KL(\nu \hat{P}_m^{x_{-m}}| \pi(\cdot|x_{-m}))\\
&=
(1-\beta)\KL(\nu \hat{P}_m^{x_{-m}}| \pi(\cdot|x_{-m}) \hat{P}_m^{x_{-m}})\\
&\leq 
(1-\beta)\KL(\nu | \pi(\cdot|x_{-m})).
\qquad\quad\qedhere
\end{align*}
\end{proof}

\begin{corollary}[MwG with independent MH updates]\label{cor:IMH}
Let $\pi$ satisfy Assumption \ref{asmp:convex_smooth_block} with $U_m(x)=0$ for $m=1,\dots,M$. Then, taking $\bar{P}_m^{x_{-m}}$ as in \eqref{eq:operator_MH} with $Q_m^{x_{-m}}(x_m,\cdot)= N(x^*_m(x_{-m}), L_m^{-1}\Id_{d_m})$, where $x^*_m(x_{-m})$ denotes the unique mode of $\pi(\cdot| x_{-m})$, we have for any $\mu \in \sP(\R^d)$
\[
\KL(\mu P^\mathrm{MwG}| \pi) \leq \left(1-\frac{1}{(\kappa^*)^{1+d_\text{max}/2}M}\right) \KL(\mu | \pi)\,,
\]
where $d_\text{max}=\max_{m=1,\dots,M}d_m$.
\end{corollary}
\begin{proof}
Since $U_m=0$ for $m\geq 1$, we have $U=U_0$ and Assumption \ref{asmp:convex_smooth_block} implies that $\pi(\cdot|x_{-m})$ is strongly log-concave and smooth with condition number upper bounded by $\kappa^*$.
Then it is well-known \citep[Equation (12)]{D19} that
$\pi(A | x_{-m})\leq (\kappa^*)^{d_m/2}Q_m^{x_{-m}}(x_m,A)\leq (\kappa^*)^{d_m/2}\bar{P}_m^{x_{-m}}(x_m,A)$ for all $x\in\R^d$, $A\in\mathcal{B}( \sX_m)$. 
The result now follows by 
part (a) of Theorem \ref{thm:MwG}, combined with Lemma \ref{lemma:minorizing}, taking $\beta=(\kappa^{*})^{-d_\text{max}/2}$.
\end{proof}
While being convenient to prove \eqref{eq:KL_contr_cond_updates}, implementing the proposal distribution $Q_m^{x_{-m}}(x_m,\cdot)= N(x^*_m(x_{-m}), L_m^{-1}\Id_{d_m})$ in Corollary \ref{cor:IMH} requires the user to compute $x^*_m(x_{-m})$ at each iteration, which makes it more cumbersome to implement, and it also results in a bound that depends exponentially on $d_\text{max}$. 
A much more common and practical choice is to use a RWM update for $P_m^{x_{-m}}$, see e.g.\ \cite[Section 4.2]{ascolani2024scalability} and \cite[Section 5.1]{qin2023spectral}.
In this case \eqref{eq:KL_contr_cond_updates} does not hold for any $\beta>0$, but explicit bounds on $\gap(\bar{P}_m^{x_{-m}})$ are available from the literature \citep{AL22}. Thus, part (b) of Theorem \ref{thm:MwG} can be directly applied.

\begin{corollary}[MwG with RWM updates]\label{thm:gap_MwG_RWM}
Let $\pi$ satisfy Assumption \ref{asmp:convex_smooth_block} with $U_m(x)=0$ for $m=1,\dots,M$.
Also, for all $m=1,\dots,M$ and $x\in\R^d$, let $\bar{P}_m^{x_{-m}}$ be as in \eqref{eq:operator_MH}
with proposal $Q_m(x_m,\cdot)=N(x_m,(L_md_m)^{-1}\Id_{d_m})$. Then
\begin{equation*}
\gap(P^\mathrm{MwG})\geq \displaystyle{\frac{C }{(\kappa^*)^2 M d_\text{max}}}\,,   
\end{equation*}
where $C>0$ is a universal constant and $d_\text{max}=\max_{m=1,\dots,M}d_m$. 
\end{corollary}
\begin{proof}
As in the proof of Corollary~\ref{cor:IMH}, $\pi(\cdot|x_{-m})$ is strongly log-concave and smooth with condition number $\kappa^*$.
Thus, by \cite[Corollary 35]{AL22} we have $\gap(\bar{P}_m^{x_{-m}})\geq C (d_m\kappa^*)^{-1}$, with $C>0$ being a universal constant (independent of $\pi$, $d_m$ and $\kappa^*$) explicitly defined therein.
The result then follows from part (b) of Theorem \ref{thm:MwG}, taking $\beta=C (\kappa^* d_\text{max})^{-1}$.
\end{proof}
In the special case of $d_m=1$ for all $m=1,\dots,M$, the bound of Corollary~\ref{thm:gap_MwG_RWM} becomes
\begin{equation*}
\gap(P^\mathrm{MwG})\geq \frac{C }{(\kappa^*)^2 d}\geq \frac{C }{\kappa^2 d}\,.   
\end{equation*}
To the best of our knowledge, this is the first explicit bound on the spectral gap of MwG schemes in general log-concave settings (see e.g.\ \cite{ascolani2024scalability} and \cite{qin2023spectral} and references therein for available results on MwG).
By arguments similar to the ones of Section \ref{subsec:tightness}, it is easy to see that the linear dependence with respect to $d$ is tight. 
On the contrary, the dependence on $\kappa$ is sub-optimal and the follow-up work \cite{secchi2025spectral} has shown that the lower bound $C/(\kappa^2 d)$ can be improved to $C/(\kappa d)$, albeit using a different choice of step-size in the RWM proposal $Q_m^{x_{-m}}$.

\section{Hit-and-Run and multivariate extensions}
\label{sec:extension_HR}

In this section we extend our main result to other sampling algorithms based on conditional updates. We focus on the Hit-and-Run sampler (H\&R), where a direction to update from is chosen at random, and we can handle the multidimensional version of it. The proof requires to generalize our main result to a blocked version of GS.

\subsection{\texorpdfstring{$\ell$}{l}-dimensional Hit-and-Run}
\label{subsec:HR_ell}

Let $P^{\mathrm{HR},\ell}$ be the Markov kernel which, at each iteration, randomly picks an $\ell$-dimensional affine subspace passing through the current point and samples from $\pi$ restricted to such space. 
This is the $\ell$-dimensional version of the H\&R scheme, which includes standard H\&R (i.e.\ Algorithm~\ref{alg:h_and_run}) as the special case with $\ell=1$. 
More formally, let $V_\ell(\R^d)$ be the Stiefel manifold, that is, the set of orthonormal $\ell$-frames: 
\begin{align}\label{eq:ell_orth_basis}
V_\ell(\R^d)=
\{
\textbf{v}=(v_1, \dots, v_\ell)\in (\mathbb{S}^{d-1})^\ell
\,:\,
v_i^\top v_j=0\hbox{ for }i\neq j
\}
\subseteq (\R^d)^\ell\,.    
\end{align}
On $V_\ell(\R^d)$ let $\sigma^{(\ell)}$ be the uniform measure, that is, the Haar measure with respect to rotations, normalized to be a probability measure~\cite{Chikuse1990}. 
Then, given $x\in\R^d$, a sample $Y\sim P^{\mathrm{HR},\ell}(x,\cdot)$ is obtained as:
\begin{enumerate}
    \item sample an orthonormal $\ell$-frame $(v_1,\dots,v_\ell)\sim \sigma^{(\ell)}$,
    \item sample $(s_1,\dots,s_\ell)\in \R^\ell$ with density proportional to $\pi(x+\sum_{i=1}^\ell s_iv_i)$ and set $Y=x+\sum_{i=1}^\ell s_iv_i$.
\end{enumerate}
If $\ell = 1$ then $V_1(\R^d)$ is the unit sphere $\mathbb{S}^{d-1} = \{ z \in \R^d \ : \ \| z \| = 1 \}$ and $\sigma^{(1)}$ is the uniform measure on it, so that $P^{\mathrm{HR},1}=P^\mathrm{HR}$ introduced in Algorithm~\ref{alg:h_and_run}. 

\begin{theorem}\label{theo:HR_ell}
Let $\pi $ satisfy Assumption~\ref{asmp:convex_smooth} and $\ell \in \{ 1, \ldots, d \}$. Then for any $\mu \in \sP(\R^d)$,
\begin{align}\label{eq:HR_KL_contr}
\KL(\mu P^{\mathrm{HR},\ell} | \pi) & \leq \left( 1 - \frac{\ell}{\kappa d} \right) \KL(\mu|\pi).
\end{align}
\end{theorem}

Contrary to our main result, Theorem \ref{theo:HR_ell} uses Assumption~\ref{asmp:convex_smooth}. This is reasonable: unlike Assumption~\ref{asmp:convex_smooth_block}, Assumption \ref{asmp:convex_smooth} is invariant under rotation, as is the kernel $P^\mathrm{HR,\ell}$. 
Also, note how the entropy contraction coefficient in \eqref{eq:HR_KL_contr} improves linearly with $\ell$. 
This is consistent with the intuition that, as $\ell$ increases, the associated Markov chain mixes faster as a larger part of $X^{(n)}$ is updated at each iteration. 
Nevertheless, sampling from $P^{\mathrm{HR},\ell}$ requires to sample from an $\ell$-dimensional distribution, which becomes more computationally expensive as $\ell$ increases, and implementing the extreme case $\ell = d$ is equivalent to being able to sample from $\pi$.

While the case $\ell = 1$ would follow directly from Theorem~\ref{theo:contraction_KL_main} after writing the H\&R kernel as a mixture of GS kernels, the case $\ell > 1$ requires the analysis of a multi-block version of GS which can be interesting in itself and which is found in Theorem~\ref{theo:GS_ell} below. The proof of Theorem~\ref{theo:HR_ell} can then be found in Section~\ref{sec:proof_HR_ell}.

As they can be of independent interest, we also state explicitly the inequality involving marginal entropies (in the flavor of Theorem~\ref{theo:functional_main}) and the corollary about the spectral gap. We do so for the kernel $P^{\mathrm{HR},d-\ell}$ as the formula are easier to parse.  
Below, if $\mathbf{v} \in V_\ell(\R^d)$, we denote by $p_\mathbf{v} : \R^d \to \R^d$ the orthogonal projection onto $\mathrm{span}(v_1, \ldots, v_\ell)$.

\begin{corollary}
\label{crl:inequality_HR_ell}
Let $\pi$ satisfy Assumption~\ref{asmp:convex_smooth} and $\ell \in \{ 1, \ldots, d \}$. Then,
\begin{enumerate}
\item[(a)] For any $\mu \in \sP(\R^d)$ we have
\begin{align*}
\int_{V_\ell(\R^d)}
\KL((p_{\mathbf{v}})_\# \mu | (p_{\mathbf{v}})_\# \pi) \,
\sigma^{(\ell)}(\d \mathbf{v})& \leq \left( 1 - \frac{d - \ell}{\kappa d} \right) \KL(\mu|\pi)\,.
\end{align*}
\item[(b)] Given $X\sim\pi$, for any $f \in L^2(\pi)$ we have
\begin{align*}
\int_{V_\ell(\R^d)}
\E\left(\mathrm{Var}(f(X) | p_{\mathbf{v}}(X))\right)\,
 \sigma^{(\ell)}(\d \mathbf{v})
& \geq \frac{d - \ell}{\kappa d} \mathrm{Var}(f(X))\,.
\end{align*}
\end{enumerate}    
\end{corollary}

\subsection{\texorpdfstring{$\ell$}{l}-dimensional random-scan Gibbs Sampler}

In order to prove the results of Section~\ref{subsec:HR_ell} we need to analyze a multi-block version of GS. Then Theorem~\ref{theo:HR_ell} will follow by writing $P^{\mathrm{HR},\ell}$ as a mixture of multi-block GS kernels. 

Given $\ell\in\{1,\dots,M\}$, consider a kernel $P^{\mathrm{GS},\ell}$ on $\R^d=\R^{d_1}\times\dots\times\R^{d_M}$ that, at each iteration, updates $\ell$ out of $M$ randomly chosen coordinates from their conditional distribution given the other $M-\ell$ coordinates.
For $S\subseteq \{1,\dots,M\}$, we denote $x_S=(x_m)_{m\in S}$ and $x_{-S}=(x_m)_{m\notin S}$, while $\pi(\cdot|x_{-S})$ denotes the conditional distribution of $X_S$ given $X_{-S}=x_{-S}$ under $X\sim \pi$.
With these notations, given $x\in\R^d$, a sample $Y\sim P^{\mathrm{GS},\ell}(x,\cdot)$ is obtained as:
\begin{enumerate}
\item pick $\ell$ out of $M$ coordinates $S\subseteq \{1,\dots,M\}$ uniformly at random without replacement,
\item sample $Y_S\sim \pi(\cdot|x_{-S})$ and set $Y_{-S}=x_{-S}$.
\end{enumerate}
If $\ell = 1$ then $P^{\mathrm{GS},\ell}=P^\mathrm{GS}$. As for the $\ell$-dimensional H\&R, the chain is expected to mix faster as $\ell$ increases, but also to be harder to implement. 

\begin{theorem}\label{theo:GS_ell}
Let $\pi$ satisfy Assumption~\ref{asmp:convex_smooth} and $\ell \in \{ 1, \ldots, M \}$. Then for any $\mu \in \sP(\R^d)$,
\begin{align*}
\KL(\mu P^{\mathrm{GS},\ell} | \pi) & \leq \left( 1 - \frac{\ell}{\kappa M} \right) \KL(\mu|\pi).
\end{align*}
\end{theorem}

This theorem is not a consequence of Theorem~\ref{theo:contraction_KL_main} , but its proof strategy is analogous. When $\ell = 1$ so that $P^{\mathrm{GS},\ell}=P^\mathrm{GS}$, Theorem \ref{theo:GS_ell} is less sharp than Theorem~\ref{theo:contraction_KL_main} since it involves the condition number $\kappa$, which is larger than $\kappa^*$. A refined analysis may replace $\kappa$ by a new condition number $\kappa(\ell) \in [\kappa^*,\kappa]$ depending on $\ell$, but we stick to $\kappa$ for simplicity.

The rest of this subsection is dedicated to the proof of Theorem~\ref{theo:GS_ell}, which will follow by adapting the steps of the proof of Theorem~\ref{theo:functional_main} to this new setting. Indeed we actually prove the analogue of Theorem~\ref{theo:functional_main}, that is,
\begin{equation}
\label{eq:to_prove_GS_ell}
\frac{1}{{M \choose \ell}} \sum_{|S|=\ell} \KL(\mu_{-S} | \pi_{-S}) \leq \left( 1 - \frac{\ell}{\kappa M} \right) \KL(\mu | \pi), 
\end{equation}
where the sum runs over all subsets of $\{ 1, \ldots, M \}$ of size $\ell$.
From this inequality the proof of Theorem~\ref{theo:GS_ell} follows in the same way that Theorem~\ref{theo:contraction_KL_main} follows from Theorem~\ref{theo:functional_main}. 

In order to prove~\eqref{eq:to_prove_GS_ell}, we need to extend Propositions \ref{prop:conclusion_estimate_U} and \ref{prop:conclusion_estimate_H} to the case of updating $\ell$ coordinates at a time. We will use this elementary identity from combinatorics: for any $m\in\{1,\dots,M\}$,
\begin{equation}
\label{eq:combinatorics}
\frac{1}{{M \choose \ell}} \sum_{|S|=\ell}  \1_{m \in S} = \frac{\ell}{M}\,.
\end{equation}
Said with words, the average number of time that an index $m$ belongs to a random subset of size $\ell$ is $\ell/M$. 

For the potential part, Proposition \ref{prop:conclusion_estimate_U} extends as follows. 

\begin{prop}
\label{prop:estimate_U_ell}
Under Assumption~\ref{asmp:convex_smooth}, let $\mu\in\sP(\R^d)$ and $T : \R^d \to \R^d$ be a measurable map. Define $t^* = 1/\kappa$ and for $S\subseteq \{1,\dots,M\}$ with $|S|=\ell$, define $\hat{T}^S : \R^d \to \R^d$ by $\hat{T}^S_{-S}(x) = x_{-S}$ while $\hat{T}^S_{S}(x)= (1-t^*) x_S + t^* T_S(x)$. Then 
\begin{equation*}
\frac{1}{{M \choose \ell}} \sum_{|S|=\ell} \U \left(  (\hat{T}^S)_\# \mu \right) \leq \left( 1 - \frac{\ell}{\kappa M} \right)\U(\mu) + \frac{\ell}{\kappa M} \U(T_\# \mu).
\end{equation*}
\end{prop}

\begin{proof}
We start by working on the potential $U$. Indeed thanks to~\eqref{eq:combinatorics}, 
\begin{equation*}
\frac{1}{{M \choose \ell}} \sum_{|S|=\ell} \nabla_{x_S} U(x)^\top (y_S - x_S)=\frac{\ell}{M} \nabla U(x)^\top (y-x), \qquad 
\frac{1}{{M \choose \ell}} \sum_{|S|=\ell} \| y_S - x_S \|^2 = \frac{\ell}{M} \| y - x \|^2.   
\end{equation*}
Thus, for any $y\in\R^d$, by $L$-smoothness, the above equality and convexity we have
\begin{align*}
\frac{1}{{M \choose \ell}} \sum_{|S|=\ell}  U(y_S, x_{-S})
& \leq 
\frac{1}{{M \choose \ell}} \sum_{|S|=\ell}
\left( U(x) + \nabla_{x_S} U(x)^\top (y_S - x_S) + \frac{L}{2} \|y_S - x_S\|^2 \right) \\
& = \frac{M-\ell}{M} U(x) + \frac{\ell}{M} \left( U(x) + \nabla U(x)^\top (y-x) + \frac{L}{2} \| y-x \|^2 \right) \\
& \leq \frac{M-\ell}{M} U(x) + \frac{\ell}{M} \left( U(y) + \frac{L - \lambda}{2} \| y-x \|^2 \right). 
\end{align*}
Then we set $y = (1-t^*) x + t^* \bar{x}$ with $t^* = 1/\kappa$ for some points $x, \bar{x}$. Noting that $U$ is $\kappa^{-1}$-convex with respect to the norm $L^{1/2}\| y-x \|$, we obtain, similarly to the computations in Lemma~\ref{lm:potential_U_key}, 
\begin{equation*}
\frac{1}{{M \choose \ell}} \sum_{|S|=\ell}  U(y_S, x_{-S}) \leq \left( 1 - \frac{\ell}{M\kappa} \right) U(x) + \frac{\ell}{M \kappa} U(\bar{x}).
\end{equation*}
Eventually we evaluate this equation for $\bar{x} = T(x)$, so that $(y_S, x_{-S}) = T(x)$, and integrate with respect to $\mu$ to get the result. 
\end{proof}

For the entropy part, Proposition \ref{prop:conclusion_estimate_H} extends as follows.
\begin{prop}
\label{prop:estimate_H_ell}
Let $\mu\in\sP(\R^d)$ and $T : \R^d \to \R^d$ be triangular, increasing and a $C^1$-diffeomorphism, and fix $t \in [0,1]$. For $S\subseteq \{1,\dots,M\}$ with $|S|=\ell$, define $\hat{T}^S : \R^d \to \R^d$ by $\hat{T}^S_{-S}(x) = x_{-S}$ while $\hat{T}^S_{S}(x)= (1-t) x_S + t T_S(x)$. Then 
\begin{equation*}
\frac{1}{{M \choose \ell}} \sum_{|S|=\ell} \sH \left(  (\hat{T}^S)_\# \mu \right) 
\leq 
\left( 1 - \frac{t \ell}{M} \right)\mathcal{H}(\mu) + \frac{t\ell}{ M} \mathcal{H}(T_\# \mu).
\end{equation*}
\end{prop}

\begin{proof}
Let $\bar{T} = (1-t) \Id + t T$. For each $x \in \R^d$, arguing as in the proof of Lemma~\ref{lm:entropy_triangular_eq} and with the help of~\eqref{eq:combinatorics} we can see that
\begin{equation*}
\frac{1}{{M \choose \ell}} \sum_{|S|=\ell} \log |\det D \hat{T}^S(x)| =  \frac{\ell}{M} \log |\det D \bar{T}(x)|.
\end{equation*}
Then we integrate with respect to $\mu$ and use~\eqref{eq:formula_entropy_pf} to get, after algebraic manipulations,
\begin{equation*}
\frac{1}{{M \choose \ell}} \sum_{|S|=\ell} \sH \left(  (\hat{T}^S)_\# \mu \right)  = \frac{M-\ell}{M} \mathcal{H}(\mu) + \frac{\ell}{M} \mathcal{H}(\bar{T}_\# \mu).
\end{equation*}
Given $\bar{T} = (1-t) \Id + t T$, we use Lemma~\ref{lm:convexity_along_KR} and this yields the conclusion.
\end{proof}

\begin{proof}[Proof of Theorem \ref{theo:GS_ell}]
If the Knothe-Rosenblatt map between $\mu$ and $\pi$ is a $C^1$ diffeomorphism, then \eqref{eq:to_prove_GS_ell} follows from Propositions \ref{prop:estimate_U_ell} and \ref{prop:estimate_H_ell} using the same arguments as in the proof of Theorem~\ref{theo:functional_main}. In the general case, \eqref{eq:to_prove_GS_ell} can be proven by approximation with the help of Proposition~\ref{prop:KR_smooth}. 
Then~\eqref{eq:to_prove_GS_ell} implies the statement in Theorem \ref{theo:GS_ell} thanks to the variational characterization of the entropy, in the same way that Theorem~\ref{theo:functional_main} implies Theorem~\ref{theo:contraction_KL_main}. 
\end{proof}

\subsection{Proof of the results for the \texorpdfstring{$\ell$}{l}-dimensional Hit-and-Run}
\label{sec:proof_HR_ell}

Before starting the proofs themselves, we need additional notations, namely a generalization of the Markov kernels $P_m$, $m=1, \ldots, M$ introduced in Section~\ref{subsec:GS_Pm}. 

If $V$ is a linear subset of $\R^d$, we write $V^\perp$ for the linear subspace orthogonal to it. We write $p_V$, $p_{V^\perp}$ for the projection onto $V$ and $V^\perp$ respectively. Given $\mu \in \sP(\R^d)$, we write $\mu_{V}( \cdot | p_{V^\perp}(x))$ for the conditional distribution of $p_V(X)$ given $p_{V^\perp}(X)=p_{V^\perp}(x)$ under $X\sim\mu$. 
If $\mu$ has a density (still written $\mu$) and $v_1, \ldots, v_\ell$ is an orthonormal frame generating $V$, then it is easy to check that the density of $\mu_V(\cdot | p_{V^\perp}(x))$ in the frame $v_1, \ldots, v_\ell$ is proportional to $\mu(p_{V^\perp}(x)+\sum_{i=1}^\ell s_i v_i)$.

If $V$ is a linear subset of $\R^d$, we write $P_V : \R^d \times \mathcal{B}(\R^d) \to [0,1]$ for the Markov kernel defined as
\begin{equation*}
P_V(x,A) = \int_V \1_{A}(p_{V^\perp}(x) + y) \, \pi_V(\d y | p_{V^\perp}(x) ).
\end{equation*}
The interpretation of the kernel $P_V(x,\cdot)$ is that it keeps the component of $x$ on $V^\perp$ unchanged while the component on $V$ is drawn according to $\pi_V(\cdot | p_{V^\perp}(x) )$. If $V$ is the subspace generated by the points in the $m$-th block, that is, by $(0, \ldots, 0, x_m, 0, \ldots,0)$ with $x_m \in \sX_m$, then $P_V$ coincides with $P_m$ introduced in Section~\ref{subsec:GS_Pm}.
We will only need later the following identity characterizing $P_V$ similar to Lemma~\ref{lm:interpretation_Pm}: if $f \in L^2(\pi)$ then with $X \sim \pi$: 
\begin{equation}
\label{eq:PV_action_f}
P_V f(x) = \E(f(X) | p_{V^\perp}(X) = p_{V^\perp}(x)).    
\end{equation}

In the sequel, to lighten the notations, if $v_1, \ldots, v_\ell$ is a set of vectors, we identify $P_{v_1, \ldots, v_\ell}$ with $P_{\mathrm{span}(v_1, \ldots, v_\ell)}$. In particular if $v$ is a vector, we write $P_v$ as a shortcut for $P_{\R v}$. With these notations, the $\ell$-dimensional H\&R kernel reads
\begin{equation*}
\label{eq:definition_HR_ell}
P^{\mathrm{HR},\ell} = \int_{V_\ell(\R^d)} P_{v_1, \ldots, v_\ell} \,  \sigma^{(\ell)}(\d \mathbf{v})\,,
\end{equation*}
where the above notation indicates a mixture of kernels, i.e.\ $P^{\mathrm{HR},\ell}(x,A) = \int_{V_\ell(\R^d)} P_{v_1, \ldots, v_\ell}(x,A) \,  \sigma^{(\ell)}(\d \mathbf{v})$ for all $x\in\R^d$ and $A\in\mathcal{B}(\R^d)$.
In the case $\ell=1$ we recall that $V_1(\R^d) = \mathbb{S}^{d-1}$ and $\sigma^{(1)}$ is the normalized uniform measure on the unit sphere.

\begin{proof}[Proof of Theorem \ref{theo:HR_ell}]
Theorem \ref{theo:HR_ell} follows from Theorem \ref{theo:GS_ell} after writing $P^{\mathrm{HR},\ell}$ as a mixture of kernels $P^{\mathrm{GS},\ell}$ over randomly rotated coordinate basis $\mathbf{v}\in V_d(\R^d)$, as we now show.

For a basis $\mathbf{v} \in V_d(\R^d)$ of $\R^d$, let us define the Markov kernel $P^{\mathrm{GS},\ell,\mathbf{v}}$ which corresponds to the $\ell$-dimensional GS, but in the basis $\mathbf{v}$. That is, 
\begin{equation*}
P^{\mathrm{GS},\ell, \mathbf{v}} = \frac{1}{{d \choose \ell}} \sum_{|S|=\ell} P_{\mathrm{span}\{ v_i \ : \ i \in S \}}.
\end{equation*}
Then using that if $\mathbf{v} \sim \sigma^{(d)}$, then $\mathbf{w} = \{ v_i \ : \ i \in S \} \in \R^{d \times \ell}$ for $S$ a subset of size $\ell$ follows marginally $\sigma^{(\ell)}$~\cite[Theorem 2.1]{Chikuse1990}, 
\begin{equation*}
\int_{V_d(\R^d)} P^{\mathrm{GS},\ell,\mathbf{v}} \,  \sigma^{(d)}(\d \mathbf{v})
=
\frac{1}{{d \choose \ell}}\sum_{|S| = \ell} 
\int_{V_d(\R^d)} 
P_{\mathrm{span}\{ v_i \ : \ i \in S \}} \,  \sigma^{(d)}(\d \mathbf{v})
= \int_{V_\ell(\R^d)} 
P_{\mathbf{w} } \,  \sigma^{(\ell)}(\d \mathbf{w})
=P^\mathrm{HR,\ell}.
\end{equation*}
In the case $\ell = 1$, it means that the H\&R kernel can be written as a mixture of (classical) GS kernels in rotated basis.
Using this expression and the convexity of $\KL$ we find 
\begin{equation*}
\KL(\mu P^{\mathrm{HR,\ell}} | \pi)
\leq
\int_{V_d(\R^d)}
\KL(\mu P^{\mathrm{GS},\ell,\mathbf{v}}| \pi) \, 
\sigma^{(d)}(\d \mathbf{v})\,.
\end{equation*}
Finally we invoke Theorem \ref{theo:GS_ell} with $d_1 = \ldots =d_M = 1$ so that $d=M$ and use the fact that, for any basis $\mathbf{v}$, the condition number $\kappa$ of $U$ is independent on the choice of the basis, so that, uniformly in $\mathbf{v}$,
\begin{equation*}
\KL(\mu P^{\mathrm{GS},\ell,\mathbf{v}} | \pi)
\leq
\left( 1 - \frac{\ell}{\kappa d} \right) \KL(\mu|\pi)\,.
\end{equation*}
Combining the last two displayed equations we get the desired result.
\end{proof}

\begin{proof}[Proof of Corollary~\ref{crl:inequality_HR_ell}]
If $\mathbf{w} = (w_1, \ldots, w_d)$ is an orthonormal basis of $\R^d$, Equation~\eqref{eq:to_prove_GS_ell} for the $\ell$-dimensional GS with $M=d$ in the basis $\mathbf{w}$ reads
\begin{equation*}
\frac{1}{{d \choose \ell}} \sum_{|S|=\ell} \KL \left( \left. (p_{\mathrm{span}\{ w_i \ : \ i \in S \}^\perp})_\# \mu \right| (p_{\mathrm{span}\{ w_i \ : \ i \in S \}^\perp})_\# \pi \right) \leq \left( 1 - \frac{\ell}{\kappa d} \right) \KL(\mu | \pi)\,. 
\end{equation*}
Then we integrate with respect to $\mathbf{w}$ according to the distribution $\sigma^{(d)}$. Using that $\{ w_i \ : \ i \in S \}$ has distribution $\sigma^{(\ell)}$ for any subset $S$ of size $\ell$~\cite[Theorem 2.1]{Chikuse1990}, 
\begin{align}
\label{eq:aux_proof_crl}
\int_{V_\ell(\R^d)}
\KL((p_{\mathbf{v}^\perp})_\# \mu | (p_{\mathbf{v}^\perp})_\# \pi) \, 
\sigma^{(\ell)}(\d \mathbf{v})& \leq \left( 1 - \frac{\ell}{\kappa d} \right) \KL(\mu|\pi),
\end{align}
where $p_{\mathbf{v}^\perp}$ is the orthogonal projection onto $\mathrm{span}\{ v_1, \ldots, v_\ell \}^\perp$.
On the other hand by linearization of~\eqref{eq:aux_proof_crl} we have for any $f \in L^2(\pi)$
\begin{align*}
\int_{V_\ell(\R^d)}
\E\left(\mathrm{Var}(f(X) | p_{\mathbf{v}^\perp}(X))\right)\,
 \sigma^{(\ell)}(\d \mathbf{v})
& \geq \frac{\ell}{\kappa d} \mathrm{Var}(f(X))\,.
\end{align*}
The conclusion of Corollary~\ref{crl:inequality_HR_ell} follows by noticing that, if $\mathbf{v} \sim \sigma^{(\ell)}$, then $p_{\mathbf{v}^\perp}$ has the same distribution as $p_{\mathbf{w}}$ with $\mathbf{w} \sim \sigma^{(d-\ell)}$~\cite[Theorem 2.1]{Chikuse1990}. Thus it is enough to apply the two last inequalities above with $d-\ell$ instead of $\ell$. 
\end{proof}

\bibliographystyle{chicago}
\bibliography{bibliography}

\appendix

\section{Bound on the relative entropy for a factorized feasible start}\label{appn}

The following is a technical result that is used in the proof of Corollary~\ref{crl:GS_mix_time}. 

\begin{prop}
\label{prop:appendix_KL_factorized}
Let $\pi$ satisfy Assumption~\ref{asmp:convex_smooth}, $x^*$ be its unique mode, and $\mu(\d x)= \bigotimes_{m=1}^M \pi(\d x_m|x^*_{-m} )$. Then $\KL(\mu|\pi) \leq d \kappa^2$. 
\end{prop}

\begin{proof}
To simplify notation, assume $\pi(x) = \exp(-U(x)) / \int \exp(-U(x))\d x$ with $x^* = \arg\min_x U(x) = 0$ and $U(0) = \inf_x U = 0$, which can be always achieved by translation and addition.

The density of $\mu$ is given by the normalization of $\exp(-V)$ with 
\begin{equation*}
V(x) = \sum_{m=1}^M U(x_{m},x_{-m}^*).  
\end{equation*}
From Assumption~\ref{asmp:convex_smooth} and the normalization we have
\begin{equation}
\label{eq:appendix_aux}
\frac{\lambda}{2} \| x \|^2 \leq \, U(x), V(x) \leq \frac{L}{2} \| x \|^2.
\end{equation}
In particular it implies $U \leq \kappa V$. Expanding the formula for the Kullback-Leibler divergence
\begin{align*}
\KL(\mu|\pi) &= \int_{\R^d} (U(x)- V(x)) \, \mu(\d x) + \log \frac{\int_{\R^d} \exp(-U(x))\d x}{\int_{\R^d} \exp(-V(x))\d x}  
\\
&\leq (\kappa - 1) \frac{\int_{\R^d} V(x) \exp(-V(x))\d x}{\int_{\R^d} \exp(-V(x))\d x} + \log \frac{\int_{\R^d} \exp(-U(x))\d x}{\int_{\R^d} \exp(-V(x))\d x}.
\end{align*}
The second term can be upper bounded  by $d \log(\kappa)/2$ using~\eqref{eq:appendix_aux}.
To bound the first term, we introduce the function $g : t \mapsto \log \int \exp(t V) $, defined for $t < 0$. This is a convex function, thus
\begin{align*}
\frac{\int_{\R^d} V(x) \exp(-V(x))\d x}{\int_{\R^d} \exp(-V(x))\d x} 
&= 
g'(-1) 
\leq 
2 \left( g\left( -\frac{1}{2} \right) - g(-1) \right) 
\\
&= 2 \log \frac{\int_{\R^d} \exp(-V(x)/2)\d x}{\int_{\R^d} \exp(-V(x))\d x} \leq d (\log(\kappa) + \log(2)),
\end{align*}
where the last equality comes again from~\eqref{eq:appendix_aux}. 
Combining the above bounds, we obtain
\begin{equation*}
\KL(\mu|\pi) \leq d \left\{  (\kappa - 1) (\log(\kappa) + \log(2)) + \frac{1}{2} \log(\kappa) \right\}.
\end{equation*}
Finally, we use $\log(\kappa) \leq \kappa - 1$ and $\log(2) \leq 1$ so that the term in brackets is upper bounded by $(\kappa-1)(\kappa+1) \leq \kappa^2 -1 \leq \kappa^2$.
\end{proof}

\end{document}